\documentclass[11pt]{article}

\usepackage{amsfonts,amssymb, latexsym, amsmath, amsthm,mathrsfs, verbatim,geometry}
\usepackage{graphics}
\usepackage{slashed}
\usepackage{url,color}
\usepackage{enumerate}
\usepackage{esint}
\usepackage{pifont}
\usepackage{upgreek}
\usepackage{times}
\usepackage{calligra}

\usepackage{hyperref}
\hypersetup{
    colorlinks,
    citecolor=red,
    filecolor=green,
    linkcolor=blue,
    urlcolor=black
}

\def\Z{\mathbb Z}

\def\be{\begin{equation}}
\def\ee{\end{equation}}
\def\bea{\begin{eqnarray}}
\def\eea{\end{eqnarray}}
\def\beas{\begin{eqnarray*}}
\def\eeas{\end{eqnarray*}}

\def\l{\lambda}
\def\pa{\partial }

\def\S{y_\ast}

\def\dz{(\delta z)}

\def\l{\lambda}

\def\lv{\left\vert}
\def\rv{\right\vert}

\def\tr{\tilde{r}}

\def\bcr{\begin{color}{red}}
\def\bcb{\begin{color}{blue}}
\def\ec{\end{color}}
\def\tom{\tilde\omega}
\def\tr{\tilde\rho}
\def\po{\pa\omega_-}
\def\pr{\pa\rho_-}
\def\ya{y_\ast}
\def\X{\mathcal X}
\def\Y{\mathcal Y}
\def\Z{\mathcal Z}

\newtheorem{theorem}{Theorem}[section]
\newtheorem{definition}[theorem]{Definition}
\newtheorem{proposition}[theorem]{Proposition}

\newtheorem{lemma}[theorem]{Lemma}
\newtheorem{remark}[theorem]{Remark}

\begin{document}

\title{Larson-Penston Self-similar Gravitational Collapse}

\author{Yan Guo\thanks{Division of Applied Mathematics, Brown University, Providence, RI 02912, USA, Email: Yan\_Guo@brown.edu.}, \ Mahir Had\v zi\'c\thanks{Department of Mathematics, University College London, London WC1E 6XA, UK. Email: m.hadzic@ucl.ac.uk.}, \ and Juhi Jang\thanks{Department of Mathematics, University of Southern California, Los Angeles, CA 90089, USA, and Korea Institute for Advanced Study, Seoul, Korea.  Email: juhijang@usc.edu.}}

\date{}
\maketitle
\abstract{Using numerical integration, in 1969 Penston~\cite{Penston1969} and Larson~\cite{Larson1969} independently discovered a self-similar solution describing the collapse of a self-gravitating asymptotically flat fluid with the isothermal equation of state $p=k\varrho$, $k>0$, and subject to Newtonian gravity. We rigorously prove the existence of such a Larson-Penston solution.}

\section{Isothermal Euler-Poisson system}

The classical model of a self-gravitating Newtonian star is given by the gravitational Euler-Poisson system. We work in three spatial dimensions and assume radial symmetry.
The unknowns are the gas density $\varrho(t,r)$, the pressure $p(t,r)$, and the radial velocity $u(t,r)$, where $t$ is the time coordinate and $r=|x|$ the radial coordinate in $\mathbb R^3$. 
The equations take the form
\begin{align}
\pa_t \varrho + \left(\pa_r+\frac2r\right)(\varrho u) & = 0, \label{E:ECONT}\\
\varrho \left(\pa_t u + u\pa_r u\right) + \pa_r p + \varrho \frac{m(r)}{r^2}
&= 0, \label{E:EMOM}\\
m(t,r) & = \int_0^r 4\pi \sigma^2 \varrho(t,\sigma)\,d\sigma. \label{E:FORCE}
\end{align}
Equation~\eqref{E:ECONT} is the continuity equation, equation~\eqref{E:EMOM} expresses the conservation of momentum, while
the term $\frac{m(r)}{r^2}$ is the radial component of the gravitational force induced by an asymptotically flat gravitational potential $\phi$ solving the Poisson equation $\Delta\phi = 4\pi\varrho$.
To complete the formulation of the problem we impose the {\em isothermal} equation of state, i.e. we let
\begin{align}\label{E:EQUATIONOFSTATE}
p = k \varrho, \ \ k>0.
\end{align}
Here $\sqrt k$ is the speed of sound and it is constant throughout the star.
We are interested in the existence of self-similar solutions to~\eqref{E:ECONT}--\eqref{E:EQUATIONOFSTATE} describing finite time gravitational collapse.  The only invariant scaling for~\eqref{E:ECONT}--\eqref{E:EQUATIONOFSTATE} is given through the transformation
\begin{align}
\varrho \mapsto  \lambda^{-2}\varrho (\frac t\l, \frac r\l), \ \ 
u  \mapsto u (\frac t\l, \frac r\l).\label{E:SCALING}
\end{align}
Motivated by this invariance, we seek a self-similar solution of~\eqref{E:ECONT}--\eqref{E:EQUATIONOFSTATE} of the form:
\begin{align}
\varrho(t,r)& =  (\sqrt{2\pi k} \, t)^{-2}\tilde\rho(y), \ \ 
u(t,r)  =  \sqrt k\tilde u(y), \ \
\end{align}
where 
\be
y : = \frac{r}{-\sqrt k t}.
\ee
It is convenient to introduce the relative velocity 
\be\label{E:RELVEL}
\tilde\omega : = \frac{\tilde u(y)+y}{y}.
\ee
Applying the above change of variables,
the Euler-Poisson system~\eqref{E:ECONT}--\eqref{E:EQUATIONOFSTATE} becomes
\begin{align}
\tr'  & = -\frac{2 y \tom\tr}{1- y^2\tom^2}(\tr-\tom) \label{E:RHOTILDEEQN}\\
\tom' & = \frac{1-3\tom}{y} + \frac{2y\tom^2}{1-y^2\tom^2}(\tr-\tom), \label{E:OMEGATILDEEQN}
\end{align}
where the derivative notation $'$ is short for $\pa_y$. A simple Taylor expansion at the origin $y=0$ and the asymptotic infinity $y\to+\infty$ shows that in order for a solution $(\tilde\rho,\tilde\omega)$ to~\eqref{E:RHOTILDEEQN}--\eqref{E:OMEGATILDEEQN} to be smooth and asymptotically flat, we must have 
\begin{align}
\tom(0) & = \frac13, \ \ \tr(0)>0 \label{E:REGZERO}\\
\tr(y)&  \sim_{y\to\infty}  y^{-2}, \ \ \lim_{y\to\infty}\tom(y)=1.\label{E:ASYMPFLAT}
\end{align}
By continuity, for any continuous solution satisfying~\eqref{E:REGZERO}--\eqref{E:ASYMPFLAT} there must exist at least one point $y_\ast$ such that 
$1-y_\ast^2\tom^2(y_\ast)=0$. At such a point  the system~\eqref{E:RHOTILDEEQN}--\eqref{E:OMEGATILDEEQN} is in general singular. 
This leads us to one of the central notions in this paper.


\begin{definition}[Sonic point]
A point $y_\ast>0$ is called a {\em sonic point} for the flow $(\tr(\cdot),\tom(\cdot))$ if 
\be\label{E:SP}
1-y_\ast^2\tom^2(y_\ast)=0.
\ee
\end{definition}


For a solution to be smooth through the sonic point $y_\ast$, it has to be the case that the sonic point is a removable singularity. Assuming smoothness, we can formally compute the Taylor coefficients of $(\tilde\rho,\tilde\omega)$ around $y_\ast$. Two possibilities emerge (see e.g.~\cite{BrWi1998}) - either
\begin{align}
\tr(y)& = \frac1{y_\ast} -\frac1{y_\ast^2} (y-y_\ast) + \frac{-y_\ast^2+6y_\ast-7}{2y_\ast^3(2y_\ast-3)} (y-y_\ast)^2 +O(|y-y_\ast|^3) 
\label{E:LPRHO}\\
\tom(y) &= \frac1{y_\ast} + \frac1{y_\ast}(1-\frac2{y_\ast}) (y-y_\ast) + \frac{-5y_\ast^2+19y_\ast-17}{2y_\ast^3(2y_\ast-3)}(y-y_\ast)^2 + O(|y-y_\ast|^3), \label{E:LPOMEGA}
\end{align}
or
\begin{align}
\tr(y)& = \frac1{y_\ast} +\frac1{y_\ast}\left(1-\frac3{y_\ast}\right) (y-y_\ast) +O((y-y_\ast)^2), \label{E:HUNTERRHO}\\
\tom(y) &= \frac1{y_\ast} + O((y-y_\ast)^2).\label{E:HUNTEROMEGA}
\end{align}

Using numerics, in 1969 in their seminal works Penston~\cite{Penston1969} and Larson~\cite{Larson1969}  independently discovered  an asymptotically flat smooth solution to~\eqref{E:RHOTILDEEQN}--\eqref{E:OMEGATILDEEQN} which satisfies the  boundary conditions~\eqref{E:REGZERO}--\eqref{E:ASYMPFLAT}. Their solution passes through a single sonic point $\S$ and conforms to the expansion of the type~\eqref{E:LPRHO}--\eqref{E:LPOMEGA}. In the literature, this solution is commonly referred to as the Larson-Penston (LP) collapsing solution. There have been numerous studies of self-similar collapse for isothermal stars in the astrophysics literature and here we only provide a brief overview. In 1977 Hunter~\cite{Hunter1977} numerically discovered a further (discrete) family of smooth self-similar solutions, commonly referred to as Hunter solutions, see also the important work of Shu~\cite{Shu1977}. The Taylor expansion of the Hunter solutions around the sonic point is of the form~\eqref{E:HUNTERRHO}--\eqref{E:HUNTEROMEGA}. A thorough analysis of the  various types of self-similar solutions is given by Whitworth and  Summers~\cite{WhSu1985}.  In 1988 Ori and Piran~\cite{OriPiran1988} gave numerical evidence that the LP collapse is the only stable self-similar solution in the above family of solutions, and therefore physically the most relevant. Brenner and Witelski~\cite{BrWi1998}, Maeda and Harada~\cite{MaedaHarada2001} reached the same conclusion performing careful numerical analysis of the collapse. The LP-solutions also play an important role in the study of so-called critical phenomena~\cite{GaGu2007} and are of central importance in astrophysics, see e.g.~\cite{Hartmann2009}. The central result of this work is the proof of existence of an LP-solution.


\begin{theorem}[Existence of a Larson-Penston self-similar collapsing solution]\label{T:MAIN}
There exists a $y_\ast\in(2,3)$ such that~\eqref{E:RHOTILDEEQN}--\eqref{E:ASYMPFLAT} possesses a real-analytic solution $(\tilde\rho,\tilde\omega)$ with a single sonic point at $y_\ast$. Moreover the solution satisfies the Larson-Penston expansion~\eqref{E:LPRHO}--\eqref{E:LPOMEGA} at $y=y_\ast$ and
\begin{align}
\tilde\rho(y)& >0, \ \ y\in[0,\infty) \label{E:DENSITYPROP}\\
-\frac23 y\le \tilde u(y)& <0, \ \ y\in[0,\infty), \label{E:VELPROP}
\end{align} 
where we recall~\eqref{E:RELVEL}.
\end{theorem} 


\begin{remark}\label{R:FRIEDMAN}
There are two known explicit solutions to~\eqref{E:RHOTILDEEQN}--\eqref{E:ASYMPFLAT}. One of them is the Friedman solution
\begin{align}\label{E:FRIEDMAN}
\tilde\rho_F(y) = \tilde\omega_F(y) \equiv \frac13
\end{align}
and the other one is the far-field solution
\begin{align}\label{E:FARFIELD}
\tilde\rho_\infty(y) = \frac1{y^2}, \ \ \tilde\omega_\infty(y) \equiv 1.
\end{align}
The Friedman solution~\eqref{E:FRIEDMAN} is the Newtonian analogue of the classical cosmological Friedman solution - it satisfies the boundary condition~\eqref{E:REGZERO}, but is not asymptotically flat. On the other hand, the far-field solution~\eqref{E:FARFIELD} is asymptotically flat, but blows up at the origin $y=0$. 
\end{remark}

If the linear equation of state~\eqref{E:EQUATIONOFSTATE} is replaced by the polytropic law $p=\varrho^\gamma$, $\gamma>1$, it is well known that there cannot exist any collapsing solutions with finite mass and energy in the regime $\gamma>\frac43$, see~\cite{DLYY}. When $\gamma=\frac43$ there exists a special class of self-similar collapsing and expanding solutions~\cite{GoWe, Makino92, FuLin,DLYY}. The nonlinear stability in the expanding case was shown in~\cite{HaJa2016-1}. When $1<\gamma<\frac43$ the authors in~\cite{GHJ} showed the existence of an infinite-dimensional class of collapsing solutions to the gravitational Euler-Poisson system.  When one considers the Euler-Poisson system with an electric (instead of gravitational) force field, the dispersive nature of the problem becomes dominant. A lot of progress has been made in recent decades, we refer the reader to~\cite{Guo1998,GeMaPa,GuPa,IoPa2013} and references therein. 


To prove Theorem~\ref{T:MAIN},
it is natural to consider the following change of variables 
\begin{align}\label{E:ZVARIABLE}
z = \frac{y}{\S}, \ \ \ \tom(y) = \omega(z), \ \ \ \tr(y) = \rho(z).  
\end{align}
The unknown sonic point $y_\ast$ is mapped to $z=1$. The system~\eqref{E:RHOTILDEEQN}--\eqref{E:OMEGATILDEEQN}
takes the form
\begin{align}
\rho'  & = -\frac{2\S^2 z \omega\rho}{1-\S^2 z^2\omega^2}(\rho-\omega) \label{E:RHOEQN}\\
\omega' & = \frac{1-3\omega}{z} + \frac{2\S^2z\omega^2}{1-\S^2z^2\omega^2}(\rho-\omega). \label{E:OMEGAEQN}
\end{align}
We shall work with this formulation for the rest of the paper and often, by abuse of terminology, refer to the point $z=1$ as the sonic point. It is now obvious from the LP sonic point expansion~\eqref{E:LPRHO}--\eqref{E:LPOMEGA} that
\be\label{E:ZEROCONST}
\omega(1) = \rho(1)= \frac1{\S},
\ee
for any solution satisfying $\omega(1),\rho(1)>0$.
If we define the infinitesimal increment
\begin{align}\label{E:INCREMENT}
\delta z: = z-1,
\end{align}
we formally assume that locally around the sonic point
\begin{align}
\rho & =  \sum_{N=0}^\infty \rho_N \dz^N, \ \ \omega  = \sum_{N=0}^\infty \omega_N \dz^N \label{E:RHOOMEGAEXPANSION} 
\end{align}
In this notation $\omega_0=\rho_0 = \frac1\S$ and~\eqref{E:LPRHO}--\eqref{E:LPOMEGA} gives us
\begin{align}\label{E:LPCONDITION}
(\rho_1,\omega_1) = (-\frac1{y_\ast}, 1-\frac2{y_\ast}), \ \ 
(\rho_2,\omega_2) = \left(\frac{-y_\ast^2+6y_\ast-7}{2y_\ast (2y_\ast-3)}, \frac{-5y_\ast^2+19y_\ast-17}{2y_\ast(2y_\ast-3)}\right)
\end{align}
For any $y_\ast>0$ we shall say that a solution of~\eqref{E:RHOEQN}--\eqref{E:OMEGAEQN} is of {\em Larson-Penston (LP)-type} if the conditions~\eqref{E:ZEROCONST} and~\eqref{E:LPCONDITION} hold. We shall prove in Theorem~\ref{T:ANALYTICITY} that for any $y_\ast>\frac32$ the LP-type conditions~\eqref{E:ZEROCONST} and \eqref{E:LPCONDITION} uniquely specify a real analytic solution in some small neighbourhood of $z=1$.
We denote this flow by $(\rho(\cdot;y_\ast),\omega(\cdot;y_\ast))$.

\subsection{Methodology}

The sonic point in the original $(t,r)$-variables corresponds to the backward cone emanating from the singularity $(0,0)$ and it takes the form $\frac r{-t} = -u(t,r)\pm\sqrt k$, $t<0$. More details on the geometric meaning of the sonic point in this context can be found for example in~\cite{BrWi1998}. Sonic points appear naturally in self-similar formulation of equations of fluid mechanics (see~\cite{BaChFe2009,ChFe2018,MRRS, Hartmann2009,JeTs} and references therein). They present a fundamental difficulty in our proof of Theorem~\ref{T:MAIN}, as we cannot use any standard ODE theory to construct a real analytic (or a $C^\infty$) solution. This is well illustrated in a recent pioneering study of sonic points for the compressible Euler system~\cite{MRRS}, where the authors use the equation of state $p=\varrho^\gamma$, $\gamma>1$. Using delicate arguments the authors~\cite{MRRS} systematically develop the existence theory for $C^\infty$ self-similar solutions of the Euler flow and such a smoothness is crucial in the proof of their nonlinear stability~\cite{MRRS2}.

The self-similar problem associated with the Euler-Poisson system leads to an ODE-system which is not autonomous (see~\eqref{E:RHOEQN}--\eqref{E:OMEGAEQN}). We also emphasise that the presence of gravity fixes exactly one invariant scale in the problem, see~\eqref{E:SCALING}.
Our proof uses in essential way dynamic invariances specific to the flow~\eqref{E:RHOEQN}--\eqref{E:OMEGAEQN}.
The sonic point separates the positive semi-axis $z\ge0$ into an inner region $[0,1]$ and an outer region $[1,\infty)$ (i.e. $[0,\S]$ and $[\S,\infty)$ in the $y$-variable). The first and the easier step is to construct an LP-type solution in the outer region satisfying the boundary condition~\eqref{E:ASYMPFLAT}. This can be done for {\em any} value of $y_\ast\in[2,3]$. The remaining key step is to find a value of $y_\ast\in[2,3]$ such that the associated LP-type solution connects $z=1$ with the {\em singular} point $z=0$ in the inner region and satisfies the boundary condition~\eqref{E:REGZERO}. More specifically, our goal is to choose $y_\ast\in[2,3]$ so that the local LP-type solution extends to the left all the way to $z=0$ and satisfies $\lim_{z\to0}\omega(z;y_\ast)=\frac13$. This motivates us to consider
\begin{align}\label{E:YDEFINTRO}
Y : = \left\{y_\ast\in[2,3]\,\big| \ \exists\, z \ \ \text{such that } \omega(z;\tilde y_\ast)=\frac13 \ \text{ for all } \ \tilde y_\ast\in[y_\ast,3]\right\}.
\end{align}
The curve $(\rho,\omega)\equiv(\frac13,\frac13)$ corresponds to the Friedman curve, see Remark~\ref{R:FRIEDMAN}. We will show that the solution curve $\omega(\cdot;\S)$ crosses the Friedman curve strictly inside the interior region and stays trapped below it for $\S$ sufficiently close to $3$. The idea is to lower the value of $\S$ to the infimum of the set $Y$ -  we set $\bar y_\ast:=\inf Y$. The idea is that $\omega(\cdot;\bar y_\ast)$ will achieve the value $\frac13$ exactly at $z=0$ and this will lead to an LP-solution.

Using the minimality of $\bar y_\ast$ it is indeed possible to show that 
the solution exists on $(0,1]$ and satisfies $\liminf_{z\to} \omega(z;\bar y_\ast)\ge \frac13$. To prove that $\lim_{z\to0}\omega(z;\bar y_\ast)=\frac13$ we use a contradiction argument in conjunction with a continuity argument. 
To explain this, it is necessary to consider the solution of the initial value problem~\eqref{E:RHOEQN}--\eqref{E:OMEGAEQN} starting from $z=0$ {\em to the right} with the initial values
\begin{align}\label{E:LEFTINITIAL}
\omega(0)=\frac13, \ \ \rho(0) = \rho_0>0.
\end{align}
Just like we did in the vicinity of the sonic point, we resort to Taylor expansion around $z=0$ to prove that (Theorem~\ref{T:ANALYTICITYLEFT}) the initial condtions~\eqref{E:LEFTINITIAL} specify a unique solution to~\eqref{E:RHOEQN}--\eqref{E:OMEGAEQN} locally to the right of $z=0$.
We denote this solution by $(\rho_-(\cdot;\rho_0),\omega_-(\cdot;\rho_0))$. 


\begin{definition}[Upper and lower solution]\label{D:ULS}
For any $y_\ast\in [2,3]$ we say that $(\rho(\cdot;y_\ast),\omega(\cdot;y_\ast))$ is an {\em upper} (resp. {\em lower}) solution at $z_0\in(0,1)$ if there exists $\rho_0>0$ such that 
\begin{align}
\rho(z_0;y_\ast) = \rho_-(z_0;\rho_0)\notag 
\end{align}
and 
\begin{align}
\omega(z_0;y_\ast)> \ (\text{resp. } <) \ \omega_-(z_0;\rho_0).\notag
\end{align}
\end{definition}


By way of contradiction we assume $\lim_{z\to0}\omega(z;\bar y_\ast)>\frac13$. 
The strategy is then the following.
\begin{itemize}

\item {\em Step 1: Upper solution.} By analysing distinct dynamic properties of the solution coming from the right $(\rho(\cdot;\bar y_\ast),\omega(\cdot;\bar y_\ast))$ and the $(\rho_-(\cdot),\omega_-(\cdot))$ emanating from the left in the region $0<z\ll1$ we show that there is a choice of $z_0\ll1$ and $\rho_0=\rho_1>\frac13$ such that $(\rho(\cdot;\bar y_\ast),\omega(\cdot;\bar y_\ast))$ is an {\em upper solution} at $z_0$ in the sense of Definition~\ref{D:ULS}.

\item {\em Step 2: Lower solution.} Using the minimality property of $\bar y_\ast = \inf Y$ and dynamic invariances associated with $(\rho_-(\cdot),\omega_-(\cdot))$ (see Lemma~\ref{L:APRIORI}) it is possible to find $y_{\ast\ast}>\bar y_\ast$ such that 
 $(\rho(\cdot;y_{\ast\ast}),\omega(\cdot;y_{\ast\ast}))$ is a {\em lower solution} at the same $z_0\ll1$ and some $\rho_0=\rho_2>0$ in the sense of Definition~\ref{D:ULS}. We emphasise that $z_0$ is the same in both steps.
 
 \item {\em Step 3: Intersection at $z=z_0$ and contradiction.} With considerable technical care and the crucial proof of strict monotonicity of the map $\rho_0\mapsto \rho_-(z;\rho_0)$ in a region $0<z_0\ll z\ll1$ (see Lemma~\ref{L:MONOTONE}),  we show that for any $y_\ast\in[\bar y_\ast, y_{\ast\ast}]$ there is a continuous map $y_\ast\mapsto \rho_0(y_\ast)$ such that 
 \[ 
 \rho(z_0;y_\ast) = \rho_-(z_0;\rho_0(y_\ast)), \ \  \rho_0(\bar y_\ast) = \rho_1, \ \ \rho_0(y_{\ast\ast}) = \rho_2.
 \]
 The Intermediate Value Theorem, Steps 1 and 2 show that there exists $y_\ast\in(\bar y_\ast,y_{\ast\ast})\subset Y$ such that $(\rho(\cdot;y_\ast),\omega(\cdot;y_\ast))$ is a solution to~\eqref{E:RHOEQN}--\eqref{E:OMEGAEQN} such that 
 $\inf_{z\in(0,1]}\omega(z;y_\ast)\ge\frac13$, which is a contradiction to the definition~\eqref{E:YDEFINTRO} of the set $Y$. 
\end{itemize}

Our work provides a general strategy to construct a solution connecting a sonic point and a singular point, such as the origin $z=0$ in this case. The crucial feature of the problem that allows us to find the solution is the contrast between the dynamics of the ``right" solution $(\rho(\cdot;\bar y_\ast),\omega(\cdot;\bar y_\ast))$ and the ``left" solution $(\rho_-(\cdot;\rho_0),\omega_-(\cdot;\rho_0))$ in the region $0<z\ll1$. This is fundamentally caused by the presence of the singular denominator $\frac1z$ on the right-hand side of~\eqref{E:OMEGAEQN}, which is a generic feature of the 3-dimensionality of the problem and radial symmetry.   

{\bf Plan of the paper.}
Section~\ref{S:TAYLORSONIC} is devoted to the proof of the local existence, uniqueness, and regularity theorems for LP-type solutions locally around the sonic point (Theorem~\ref{T:ANALYTICITY}) and around the centre $z=0$ (Theorem~\ref{T:ANALYTICITYLEFT}). In Section~\ref{S:OUTERREGION} we analyse the solution in the outer region $z>1$. The main statement is Proposition~\ref{P:FGE} which states that for any $\S\in[2,3]$ there exists a global forward-in-$z$ solution to our problem starting from the sonic point $z=1$ (i.e. $y=\S$). The most difficult part of our work is the analysis of the inner region $z\in[0,1)$ and it is contained in Section~\ref{S:INNERREGION}. In Section~\ref{SS:SC} we obtain various continuity results for the LP-type flows, including most importantly the upper semi-continuity of the so-called sonic time, see Proposition~\ref{P:LOWERSEMICONT}. In Section~\ref{SS:SETY} we introduce the crucial set $Y$ and show that the LP-type flow associated with $\bar y_\ast=\inf Y$ starting from $z=1$ exists all the way to $z=0$, see Proposition~\ref{P:EXISTENCEY}.
Qualitative properties of the flow $(\rho_-,\omega_-)$ are investigated in Section~\ref{SS:LEFT}. Finally, the key intersection argument and the proof that $\lim_{z\to0^+}\omega(z;\bar y_\ast)=\frac13$ is presented in Section~\ref{S:INTERSECTION}, see Propositions~\ref{P:ONETHIRD} and~\ref{P:GOODPROPERTIES}. Finally, in Section~\ref{S:MAIN} we prove the main result - Theorem~\ref{T:MAIN}.

\section{Local well-posedness near the sonic point and the origin}\label{S:TAYLORSONIC}

\subsection{Existence, uniqueness, and regularity near the sonic point}
Recalling~\eqref{E:RHOOMEGAEXPANSION}, 
our goal is to compute a recursive relation that expresses the vector $(\rho_N,\omega_N)$ in terms of $\rho_0,\dots,\rho_{N-1}, \omega_0,\dots, \omega_{N-1}$.
For a given function $f$ we shall write $(f)_M$ to mean the $M$-th Taylor coefficient in the expansion of $f$ around the sonic point $z=1$. In particular, 
\begin{align}
(\omega^2)_M & = \sum_{k+\ell=M} \omega_k\omega_\ell \notag \\
(\omega\rho(\rho-\omega))_M & = \sum_{k+\ell + m = M}\omega_k\rho_\ell (\rho_m-\omega_m) \notag \\
(\omega^2(\rho-\omega))_M & = \sum_{k+\ell + m = M}\omega_k\omega_\ell (\rho_m-\omega_m),\notag 
\end{align}
where the summation implicitly runs over all non-negative integers satisfying the indicated constraint.

To compute the Taylor coefficients in~\eqref{E:RHOOMEGAEXPANSION}, we first multiply~\eqref{E:RHOEQN}--\eqref{E:OMEGAEQN} by $(1-\S^2z^2\omega^2)$
\begin{align}
\omega' (1-\S^2(1+\delta z)^2\omega^2) - \left(1-3\omega\right)(1-\S^2z^2\omega^2)\frac1{1+\delta z} -
2\S^2(1+\delta z)\omega^2(\rho-\omega) & =0, \label{E:OMEGAEQN1}\\
\rho' (1-\S^2(1+\delta z)^2\omega^2) +2\S^2 (1+\delta z) \omega\rho(\rho-\omega) &=0, \label{E:RHOEQN1}
\end{align}
where we have written $z$ in the form $1+\delta z$. 


\begin{lemma}\label{L:TAYLOREXP0}
For any $N\ge0$ the following formulas hold:
\begin{align}
& (N+1)\rho_{N+1} - \S^2 \left(\sum_{k+\ell = N}(k+1)\rho_{k+1}(\omega^2)_\ell + 2\sum_{k+\ell = N-1}(k+1)\rho_{k+1}(\omega^2)_\ell
 + \sum_{k+\ell = N-2}(k+1)\rho_{k+1}(\omega^2)_\ell\right) \notag \\ 
&  + 2 \S^2 \left(\left(\omega\rho(\rho-\omega)\right)_N+\left(\omega\rho(\rho-\omega)\right)_{N-1} \right) = 0 \label{E:TAYLORRHOFORMULA} \\
& 0  = (N+1)\omega_{N+1} - \S^2 \left(\sum_{k+\ell = N}(k+1)\omega_{k+1}(\omega^2)_\ell + 2\sum_{k+\ell = N-1}(k+1)\omega_{k+1}(\omega^2)_\ell
 + \sum_{k+\ell = N-2}(k+1)\omega_{k+1}(\omega^2)_\ell\right) \notag \\
 & - (-1)^N + 3  \sum_{k+m=N}\omega_k(-1)^m 
 + \S^2 \left(\sum_{\ell+m=N-2}(-1)^m(\omega^2)_\ell+2\sum_{\ell+m=N-1}(-1)^m(\omega^2)_\ell+\sum_{\ell+m=N}(-1)^m(\omega^2)_\ell\right) \notag \\
 & - 3\S^2 \left(\sum_{k+\ell+m=N-2}(-1)^m\omega_k(\omega^2)_\ell
+ 2 \sum_{k+\ell+m=N-1}(-1)^m\omega_k(\omega^2)_\ell + \sum_{k+\ell+m=N}(-1)^m\omega_k(\omega^2)_\ell\right) \notag \\
& - 2\S^2 \left(\left(\omega^2(\rho-\omega)\right)_N+\left(\omega^2(\rho-\omega)\right)_{N-1} \right).\label{E:TAYLOROMEGAFORMULA} 
\end{align}
\end{lemma}


\begin{proof}
We plug in~\eqref{E:RHOOMEGAEXPANSION}into~\eqref{E:RHOEQN1} 
and obtain
\begin{align}
0 & = \left(\sum_{k=0}^\infty k \rho_k \dz^{k-1}\right)\left(1 - \S^2\sum_{\ell=0}^\infty(\omega^2)_\ell\left(\dz^{\ell+2}+2\dz^{\ell+1}+\dz^\ell\right) \right) \notag \\
& \ \ \ \ + 2 \S^2 \sum_{k=0}^\infty \left(\omega\rho(\rho-\omega)\right)_k\left(\dz^{k+1}+\dz^k\right) \notag \\
& = \sum_{N=0}^\infty (N+1)\rho_{N+1}\dz^N \notag \\
& \ \ \ \  - \S^2 \sum_{N=0}^\infty \left(\sum_{k+\ell = N}(k+1)\rho_{k+1}(\omega^2)_\ell + 2\sum_{k+\ell = N-1}(k+1)\rho_{k+1}(\omega^2)_\ell
 + \sum_{k+\ell = N-2}(k+1)\rho_{k+1}(\omega^2)_\ell\right) \dz^N \notag \\
 & \ \ \ \ + 2 \S^2 \sum_{N=0}^\infty \left(\left(\omega\rho(\rho-\omega)\right)_N+\left(\omega\rho(\rho-\omega)\right)_{N-1} \right) \dz^N,
\end{align}
where, by definition $\rho_k=\omega_k=0$ for $k<0$. Equating the coefficients above, we conclude that for any non-negative $N$ we have
\begin{align}
& (N+1)\rho_{N+1} - \S^2 \left(\sum_{k+\ell = N}(k+1)\rho_{k+1}(\omega^2)_\ell + 2\sum_{k+\ell = N-1}(k+1)\rho_{k+1}(\omega^2)_\ell
 + \sum_{k+\ell = N-2}(k+1)\rho_{k+1}(\omega^2)_\ell\right) \notag \\ 
&  + 2 \S^2 \left(\left(\omega\rho(\rho-\omega)\right)_N+\left(\omega\rho(\rho-\omega)\right)_{N-1} \right) = 0, \label{E:TAYLORRHO1}
\end{align}
which is precisely~\eqref{E:TAYLORRHOFORMULA}.

To prove~\eqref{E:TAYLOROMEGAFORMULA}
we first note that 
\[
\frac1{1+\delta z}= \sum_{m=0}^\infty (-1)^m\dz^m
\]
and therefore
\begin{align}
& \left(1-3\omega\right)(1-\S^2z^2\omega^2)\frac1{1+\delta z}  \notag \\
& = \left(1-3\sum_{k=0}^\infty\omega_k\dz^k\right)\left(1 - \S^2\sum_{\ell=0}^\infty(\omega^2)_\ell\left(\dz^{\ell+2}+2\dz^{\ell+1}+\dz^\ell\right) \right)  \sum_{m=0}^\infty (-1)^m\dz^m \notag \\
& = \Big(1 - 3\sum_{k=0}^\infty\omega_k\dz^k -  \S^2\sum_{\ell=0}^\infty(\omega^2)_\ell\left(\dz^{\ell+2}+2\dz^{\ell+1}+\dz^\ell\right) \notag\\
&  \ \ \ \ + 3 \S^2 \sum_{k=0}^\infty\omega_k\dz^k \sum_{\ell=0}^\infty(\omega^2)_\ell\left(\dz^{\ell+2}+2\dz^{\ell+1}+\dz^\ell\right)  \Big)
\times \sum_{m=0}^\infty (-1)^m\dz^m \notag \\
& = \sum_{N=0}^\infty (-1)^N\dz^N - 3 \sum_{N=0}^\infty \sum_{k+m=N}\omega_k(-1)^m \dz^N \notag \\
& \ \ \ \ - \S^2 \left(\sum_{\ell+m=N-2}(-1)^m(\omega^2)_\ell+2\sum_{\ell+m=N-1}(-1)^m(\omega^2)_\ell+\sum_{\ell+m=N}(-1)^m(\omega^2)_\ell\right)\dz^N \notag \\
& \ \ \ \ + 3\S^2 \left(\sum_{k+\ell+m=N-2}(-1)^m\omega_k(\omega^2)_\ell
+ 2 \sum_{k+\ell+m=N-1}(-1)^m\omega_k(\omega^2)_\ell + \sum_{k+\ell+m=N}(-1)^m\omega_k(\omega^2)_\ell\right) \dz^N. \label{E:INTERMEDIATE}
\end{align}
We plug in~\eqref{E:RHOOMEGAEXPANSION}into~\eqref{E:OMEGAEQN1}
and obtain 
\begin{align}
0 & = \left(\sum_{k=0}^\infty k \omega_k \dz^{k-1}\right)\left(1 - \S^2\sum_{\ell=0}^\infty(\omega^2)_\ell\left(\dz^{\ell+2}+2\dz^{\ell+1}+\dz^\ell\right) \right) \notag \\
& \ \ \ \  -\sum_{N=0}^\infty \left(\left(1-3\omega\right)(1-\S^2z^2\omega^2)\frac1{1+\delta z} \right)_N\dz^N \notag \\
& \ \ \ \ - 2 \S^2 \sum_{N=0}^\infty \left(\left(\omega^2(\rho-\omega)\right)_N+\left(\omega^2(\rho-\omega)\right)_{N-1} \right) \dz^N \notag \\
& = \sum_{N=0}^\infty (N+1)\omega_{N+1}\dz^N \notag \\
& \ \ \ \  - \S^2 \sum_{N=0}^\infty \left(\sum_{k+\ell = N}(k+1)\omega_{k+1}(\omega^2)_\ell + 2\sum_{k+\ell = N-1}(k+1)\omega_{k+1}(\omega^2)_\ell
 + \sum_{k+\ell = N-2}(k+1)\omega_{k+1}(\omega^2)_\ell\right) \dz^N \notag \\
& \ \ \ \  -\sum_{N=0}^\infty \left(\left(1-3\omega\right)(1-\S^2z^2\omega^2)\frac1{1+\delta z} \right)_N\dz^N \notag \\
& \ \ \ \ - 2 \S^2 \sum_{N=0}^\infty \left(\left(\omega^2(\rho-\omega)\right)_N+\left(\omega^2(\rho-\omega)\right)_{N-1} \right) \dz^N 
\notag 
\end{align}
Equating the coefficients and using~\eqref{E:INTERMEDIATE}, we conclude that for any non-negative $N$ we have
\begin{align}
& 0  = (N+1)\omega_{N+1} - \S^2 \left(\sum_{k+\ell = N}(k+1)\omega_{k+1}(\omega^2)_\ell + 2\sum_{k+\ell = N-1}(k+1)\omega_{k+1}(\omega^2)_\ell
 + \sum_{k+\ell = N-2}(k+1)\omega_{k+1}(\omega^2)_\ell\right) \notag \\
 & - (-1)^N + 3  \sum_{k+m=N}\omega_k(-1)^m 
 + \S^2 \left(\sum_{\ell+m=N-2}(-1)^m(\omega^2)_\ell+2\sum_{\ell+m=N-1}(-1)^m(\omega^2)_\ell+\sum_{\ell+m=N}(-1)^m(\omega^2)_\ell\right) \notag \\
 & - 3\S^2 \left(\sum_{k+\ell+m=N-2}(-1)^m\omega_k(\omega^2)_\ell
+ 2 \sum_{k+\ell+m=N-1}(-1)^m\omega_k(\omega^2)_\ell + \sum_{k+\ell+m=N}(-1)^m\omega_k(\omega^2)_\ell\right) \notag \\
& - 2\S^2 \left(\left(\omega^2(\rho-\omega)\right)_N+\left(\omega^2(\rho-\omega)\right)_{N-1} \right), \notag 
\end{align}
which is precisely~\eqref{E:TAYLOROMEGAFORMULA}. 
\end{proof}


\begin{lemma}
The coefficients $(\rho_i,\omega_i)$, $i=0,1$ satisfy the following formulas:
\begin{align}
\rho_0 = \omega_0 = \frac1{y_\ast}, \notag  \\
\text{Either}
\ \ (\rho_1,\omega_1) = (-\omega_0, 1-2\omega_0) \ \ \text{ or } \ \ (\rho_1,\omega_1) = (1-3\omega_0, 0)  \notag 
\end{align}
\end{lemma}

\begin{proof}
Letting $N=0$ in~\eqref{E:TAYLORRHOFORMULA}--\eqref{E:TAYLOROMEGAFORMULA} respectively
we obtain
\begin{align}
\left(1-\S^2\omega_0^2\right)\rho_1+2\S^2\omega_0^2(\rho_0-\omega_0) & = 0 \notag \\
\left(1-\S^2\omega_0^2\right)\left(\omega_1-1+3\omega_0\right) - 2\S^2\omega_0^2(\rho_0-\omega_0) & =0.\notag 
\end{align}
This is of course consistent with the $0$-th order sonic point condition~\eqref{E:ZEROCONST}. We now let $N=1$ in
~\eqref{E:TAYLORRHOFORMULA}--\eqref{E:TAYLOROMEGAFORMULA} and obtain respectively
\begin{align}
2\omega_1\left(\frac{\rho_1}{\omega_0}+1\right) & = 0 \label{E:SONICRHO1}\\
\frac{\omega_1^2}{\omega_0}-\frac{\omega_1}{\omega_0}+ 3\omega_1+3\omega_0+\rho_1-1 & =0.\label{E:SONICOMEGA1}
\end{align}
From~\eqref{E:SONICRHO1} we have two possibilities: either $\rho_1=-\omega_0$ or $\omega_1=0$.
If $\rho_1=-\omega_0$ we obtain from~\eqref{E:SONICOMEGA1}
\begin{align}
0 & = \frac{\omega_1^2}{\omega_0}-\frac{\omega_1}{\omega_0}+ 3\omega_1+2\omega_0-1 
 = \left(\omega_1 + 2\omega_0-1\right)\left(\frac{\omega_1}{\omega_0}+1\right). \notag 
\end{align}
In this case 
$
\omega_1 = 1- 2 \omega_0
$
(which corresponds to the Larson-Penston solution) or
$
\omega_1 = - \omega_0.$ We disregard the latter possibility as it corresponds to a trivial solution that appears  due to
multiplication of~\eqref{E:RHOEQN}--\eqref{E:OMEGAEQN} by  $1-\S^2z^2\omega^2$. 
If on the other hand $\omega_1=0$  we obtain 
$
\rho_1 = 1-3\omega_0. \notag 
$
from~\eqref{E:SONICOMEGA1}.
\end{proof}

By a careful tracking of top-order terms in Lemma~\ref{L:TAYLOREXP0} we will next express $(\rho_N,\omega_N)$ as a function of the Taylor coefficients with indices less or equal to $N-1$.


\begin{lemma}\label{L:MATRIXA}
Let $N\ge2$. Then the following identity holds:
\begin{align}
\mathcal A_N(\omega_0,\omega_1,\rho_1) \begin{pmatrix} \rho_N \\ \omega_N \end{pmatrix} = \begin{pmatrix} \mathcal F_N \\ \mathcal G_N \end{pmatrix}, \notag 
\end{align}
where
\begin{align}\label{E:MATRIXIDENTITY}
\mathcal A_N(\omega_0,\omega_1,\rho_1) = 
\begin{pmatrix} -2N+2-2N\frac{\omega_1}{\omega_0} &  -\frac{2\rho_1}{\omega_0}-2 \\
-2 & -2N-4+\frac2{\omega_0}-(2N+2)\frac{\omega_1}{\omega_0} \end{pmatrix}
\end{align}
and 
\begin{align}
\mathcal F_N & = \mathcal F_N[\rho_0,\omega_0;\rho_1,\omega_1;\dots \rho_{N-1},\omega_{N-1}] \notag \\
\mathcal G_N & = \mathcal G_N[\rho_0,\omega_0;\rho_1,\omega_1;\dots \rho_{N-1},\omega_{N-1}] \notag
\end{align} 
are nonlinear polynomials of the first $N-1$ Taylor coefficients given explicitly 
by the formulas~\eqref{E:FNDEF} and~\eqref{E:GNDEF} below.
\end{lemma}


\begin{proof}
We first isolate all the coefficients in~\eqref{E:TAYLORRHO1} that contain contributions from vectors $(\rho_{N+1},\omega_{N+1})$ and $(\rho_{N},\omega_{N})$. 
For $N\ge 2$ we obtain
\begin{align}
0 & =   (N+1)\rho_{N+1} - 
\S^2\left((N+1)\rho_{N+1}(\omega^2)_0 + 2N \rho_N\omega_0\omega_1
+2 \rho_1 \omega_0\omega_N + 2 N \rho_N (\omega^2)_0\right) \notag \\
& \ \ \ \ + 2\S^2 \omega_0\rho_0(\rho_N-\omega_N) - \mathcal F_N \notag \\
& =  (N+1)\rho_{N+1} - 
\omega_0^{-2}\left((N+1)\rho_{N+1}\omega_0^2 + 2N \rho_N\omega_0\omega_1
+2 \rho_1 \omega_0\omega_N + 2 N \rho_N \omega_0^2\right) \notag \\
& \ \ \ \ + 2\omega_0^{-2} \omega_0\rho_0(\rho_N-\omega_N) - \mathcal F_N \notag \\
& = \left(-2N+2 - 2N\frac{\omega_1}{\omega_0}\right)\rho_N + \left(-\frac{2\rho_1}{\omega_0}-2\right)\omega_N - \mathcal F_N
\label{E:MATRIXONE}
\end{align}
where we have used
\begin{align}
\mathcal F_N = & \S^2 \left(\sum_{k+\ell = N \atop 0<k<N-1}(k+1)\rho_{k+1}(\omega^2)_\ell 
+ \sum_{m+n=N \atop 0<m<N}\rho_1\omega_m\omega_n  + 2\sum_{k+\ell = N-1 \atop k<N-1}(k+1)\rho_{k+1}(\omega^2)_\ell
 + \sum_{k+\ell = N-2}(k+1)\rho_{k+1}(\omega^2)_\ell\right) \notag \\
 & + 2\S^2 \left(\sum_{k+\ell+n=N \atop n<N }\omega_k\rho_\ell(\rho_n-\omega_n)+\left(\omega\rho(\rho-\omega)\right)_{N-1} \right). \label{E:FNDEF}
\end{align}

We now isolate all the coefficients in~\eqref{E:TAYLOROMEGAFORMULA} that contain contributions from vectors $(\rho_{N+1},\omega_{N+1})$ and $(\rho_{N},\omega_{N})$. For $N\ge2$ we obtain
\begin{align}
0 & = (N+1)\omega_{N+1} -\S^2\left((N+1)\omega_{N+1}\omega_0^2 + 2N \omega_0\omega_1\omega_N+2\omega_0\omega_1\omega_N + 2N \omega_0^2 \omega_N\right)  \notag \\
& \ \ \ \ - 2\S^2 \omega_0^2\left(\rho_N-\omega_N\right) + 3\omega_N + \S^2\left(2\omega_0\omega_N-9\omega_0^2\omega_N\right) - \mathcal G_N \notag \\
& = -2\rho_N +\left(-2N-4+\frac2{\omega_0}-(2N+2)\frac{\omega_1}{\omega_0}\right)\omega_N - \mathcal G_N,
\label{E:MATRIXTWO}
\end{align}
where
\begin{align}
\mathcal G_N = & \S^2 \left(\sum_{k+\ell = N \atop 0<k<N-1}(k+1)\omega_{k+1}(\omega^2)_\ell 
+ \sum_{m+n=N \atop 0<m<N}\omega_1\omega_m\omega_n  + 2\sum_{k+\ell = N-1 \atop k<N-1}(k+1)\omega_{k+1}(\omega^2)_\ell
 + \sum_{k+\ell = N-2}(k+1)\omega_{k+1}(\omega^2)_\ell\right) \notag \\
 & - 2\S^2 \left(\sum_{k+\ell+n=N \atop n<N }\omega_k\omega_\ell(\rho_n-\omega_n)+\left(\omega^2(\rho-\omega)\right)_{N-1} \right) \notag \\
 & + (-1)^N - 3 \sum_{k+m=N\atop k<N}\omega_k(-1)^m \notag \\
 & - \S^2 \left(\sum_{\ell+m=N-2}(-1)^m(\omega^2)_\ell+2\sum_{\ell+m=N-1}(-1)^m(\omega^2)_\ell
 +\sum_{\ell+m=N \atop \ell<N}(-1)^m(\omega^2)_\ell + \sum_{k+n=N\atop 0<k<N}\omega_k\omega_n\right) \notag \\
 &+ 3\S^2 \Big(\sum_{k+\ell+m=N-2}(-1)^m\omega_k(\omega^2)_\ell
+ 2 \sum_{k+\ell+m=N-1}(-1)^m\omega_k(\omega^2)_\ell \notag \\
& \ \ \ \ \qquad  + \sum_{k+\ell+m=N \atop k\neq N, \ell\neq N}(-1)^m\omega_k(\omega^2)_\ell
+ \sum_{k+\ell+m=N}(-1)^m\omega_0\sum_{a+b=N\atop 0<a<N }\omega_a\omega_b\Big) \label{E:GNDEF}.
\end{align}
Finally, equations~\eqref{E:MATRIXONE} and~\eqref{E:MATRIXTWO} give~\eqref{E:MATRIXIDENTITY}.
\end{proof}


\begin{lemma}\label{L:LP}
Let $y_\ast>0$ be given.
Then the matrix 
\begin{align}
\mathcal A_N^{LP} : = \mathcal A_N(\omega_0,-\omega_0,1-2\omega_0) \notag 
\end{align}
associated with an LP-type solution is singular if and only if 
\begin{align}\label{E:LPSING}
y_\ast =1 \ \ \text{ or } \ \ \S = \frac{N+1}{N}, \ \ \text{ for some } N\ge2.
\end{align}
As a consequence, the matrix $\mathcal A_N^{LP}$ is invertible for any $\S>\frac32$ for all $N\ge2$.
\end{lemma}


\begin{proof}
Since in the case of LP-type solutions (i.e. $(\rho_1,\omega_1) = (-\omega_0, 1-2\omega_0)$) the matrix $\mathcal A_N$ takes the form
\begin{align}
\mathcal A_N^{LP} = 
2\begin{pmatrix} N(1-\frac1{\omega_0})+1 &  0 \\
-1 &  N(1-\frac1{\omega_0}) \end{pmatrix} \notag 
\end{align}
In particular
\begin{align}
\frac14\det \mathcal A^{LP}_N & = N\left(N(1-\frac1{\omega_0})+1 \right)\left(1-\frac1{\omega_0}\right). \notag 
\end{align}
Therefore,  the matrix $\mathcal A_N^{LP}$ is singular when $\omega_0=1$ or $\omega_0=\frac N{N+1}$, $N\ge2$. This  together with~\eqref{E:ZEROCONST} implies~\eqref{E:LPSING}. Since $ \frac{N}{N+1}\ge\frac23$ for all $N\ge2$ $\mathcal A_N^{LP}$ is invertible for
any $0<\omega_0<\frac23$ and $N\ge2$. 
\end{proof}

\begin{remark}[Hunter solutions]
In the case of Hunter-type solutions (i.e. $ (\rho_1,\omega_1) = (1-3\omega_0, 0)$) the matrix $\mathcal A_N$ takes the form
\begin{align}
\mathcal A_N^{H} = 
2\begin{pmatrix} -N+1 &  2-\frac1{\omega_0} \\
-1 & -N-2+\frac1{\omega_0} \end{pmatrix} \notag 
\end{align}
In particular
\begin{align}
\frac14\det \mathcal A^{H}_N &=  (N-1)(N+2-\frac1{\omega_0})+2-\frac1{\omega_0} \notag  \\
& = N (N+1-\frac1{\omega_0}) \label{E:DETHUNTER}
\end{align}
It follows that the matrix $\mathcal A_N^H$ is singular if and only if $\S=N+1$ for some $N\ge2$.
\end{remark}

For any $\S>\frac32$ consider the formal series~\eqref{E:RHOOMEGAEXPANSION} of LP-type, i.e. with conditions~\eqref{E:ZEROCONST} and~\eqref{E:LPCONDITION} satisfied.
By Lemmas~\ref{L:MATRIXA} and~\ref{L:LP} we have the explicit relations
\begin{align}
\rho_N & = \frac1{2\left(N(1-\frac1{\omega_0})+1\right)} \mathcal F_N \label{E:RHONFN}\\
\omega_N & = \frac1{2N(1-\frac1{\omega_0})} \mathcal G_N
+ \frac1{2N(1-\frac1{\omega_0})\left(N(1-\frac1{\omega_0})+1\right)} \mathcal F_N \label{E:OMEGANFNGN}
\end{align}

The assumption $\omega_0<\frac23$ (recall $\omega_0=\frac1{\S}$ by~\eqref{E:ZEROCONST}) implies that there exists a universal constant $\alpha>0$ such that 
\begin{align}
\lv \rho_N\rv & \le \frac{\alpha}{N\left(\frac23-\omega_0\right)} \lv \mathcal F_N\rv \label{E:RHONBOUNDONE}\\
\lv \omega_N\rv & \le \frac{\alpha}{N\left(\frac23-\omega_0\right)} \left(\lv \mathcal G_N\rv + \frac1N  \lv \mathcal F_N\rv \right).
\label{E:OMEGANBOUNDONE}
\end{align}

Our goal is to show 
that the formal power series $\sum_{N=0}^\infty \rho_N\dz^N$, $\sum_{N=0}^\infty \omega_N\dz^N$ converge. To that end we need some simple technical bounds which will be important in establishing convergence later on.
\begin{lemma}\label{L:COMB}
There exists a constant $c>0$ such that for all $N\in\mathbb N$ the following bounds hold
\begin{align}
 \sum_{k=1}^{N-1}\frac{1}{k^2(N-k)^2} &\le c N^{-2} \label{E:COMB1} \\
\sum_{k+\ell + m = N \atop 0<k,\ell,m}\frac1{k^2\ell^2m^2} &\le c N^{-2}, \label{E:COMB2} \\
\sum_{k+\ell + m = N \atop 0<k,\ell,m}\frac1{k\ell^2m^2} &\le c N^{-1}. \label{E:COMB3}
\end{align}
\end{lemma}


\begin{proof}
We note that 
\begin{align}
 \sum_{k=1}^{N-1}\frac{1}{k^2(N-k)^2}
 =  \sum_{k=1}^{N-1}\frac1{N^2}\left(\frac{1}{k}+\frac{1}{N-k}\right)^2 \le \frac2{N^2}\sum_{k=1}^\infty \frac1{k^2} \lesssim N^{-2}\notag 
\end{align}
and this proves~\eqref{E:COMB1}.
Next
\begin{align}
\sum_{k+\ell + m = N \atop 0<k,\ell,m}\frac1{k^2\ell^2m^2} 
& \le \sum_{k=1}^{N-1}\frac1{k^2}\sum_{\ell+m=N-k \atop 0<\ell,m} \frac1{\ell^2m^2}  
 \lesssim \sum_{k=1}^N\frac1{k^2}\frac1{(N-k)^2} 
 \lesssim  N^{-2}, \notag 
\end{align}
where we have used the already established bound~\eqref{E:COMB1} in each of the last two lines above. This proves~\eqref{E:COMB2}. 
Finally, 
\begin{align}
\sum_{k+\ell + m = N \atop 0<k,\ell,m}\frac1{k\ell^2m^2} 
 & \le \sum_{k=1}^N\frac1{k}\sum_{\ell+m=N-k \atop 0<\ell,m} \frac1{\ell^2m^2}  
 \lesssim \sum_{k=1}^N\frac1{k}\frac1{(N-k)^2} \notag \\
& =  \sum_{k=1}^N\frac1N \left(\frac1k+\frac1{N-k}\right)\frac1{N-k} 
\lesssim \frac1N, \notag 
\end{align}
and this completes the proof of~\eqref{E:COMB3}. 
\end{proof}


\begin{lemma}\label{L:KEYINDUCTION}
Let $\S>\frac32$. Then there exists a constant $C_\ast>0$ such that if $C>C_\ast$ and the following assumptions hold:
\begin{align}
\lv \rho_k \rv & \le \frac{C^{k-1}}{k^2}, \ \ 1\le k \le N-1 \label{E:ASS1}\\
\lv \omega_k \rv & \le \frac{C^{k-1}}{k^2}, \ \ 1\le k \le N-1, \label{E:ASS2}
\end{align}
then 
\begin{align}
\lv \mathcal F_N\rv & \le \frac{\beta}{\omega_0^2} \frac{C^{N-2}}{N}, \label{E:FNBOUND} \\ 
\lv \mathcal G_N\rv & \le  \frac{\beta}{\omega_0^2} \frac{C^{N-2}}{N}. \label{E:GNBOUND} 
\end{align}
for some universal constant $\beta>0$.
\end{lemma}


\begin{proof}
We start with the estimate on the first term in the definition~\eqref{E:FNDEF} of $\mathcal F_N$.
\begin{align}
\sum_{k+\ell = N \atop 0<k<N-1}(k+1)\rho_{k+1}(\omega^2)_\ell
&\lesssim C^{N-2}\sum_{k+\ell + m = N \atop 0<k<N-1, 0<\ell,m}\frac1{(k+1)\ell^2m^2}
+  C^{N-2} \sum_{k=1}^{N-1}\frac1{k+1}\frac1{(N-k)^2}\notag \\
& \lesssim \frac{C^{N-2}}{N}, \label{E:BOUNDONE}
\end{align}
where we have used Lemma~\ref{L:COMB}, estimate~\eqref{E:COMB3}. By the same argument we have
\begin{align}
\lv 2\sum_{k+\ell = N-1 \atop k<N-1}(k+1)\rho_{k+1}(\omega^2)_\ell\rv & \lesssim  \frac{C^{N-3}}{N}, \ \ 
\lv \sum_{k+\ell = N-2}(k+1)\rho_{k+1}(\omega^2)_\ell\rv  \lesssim  \frac{C^{N-4}}{N} \notag 
\end{align}
Similarly, using Lemma~\ref{L:COMB}, bound~\eqref{E:COMB2}
\begin{align}
\lv \sum_{k+\ell+n=N \atop n<N }\omega_k\rho_\ell(\rho_n-\omega_n)\rv & \lesssim  \frac{C^{N-3}}{N^2}, \ \ 
\lv \left(\omega\rho(\rho-\omega)\right)_{N-1}\rv  \lesssim  \frac{C^{N-4}}{N^2} .\label{E:BOUNDTHREE}
\end{align}
Finally 
\begin{align}
\lv \sum_{m+n=N \atop 0<m<N}\rho_1\omega_m\omega_n  \rv
\le C^{N-2}\sum_{m=1}^{N-1} \frac1{m^2(N-m)^2} \lesssim \frac{C^{N-2}}{N^2}, \label{E:BOUNDFOUR}
\end{align}
where we have used~\eqref{E:COMB1}.
From the definition~\eqref{E:FNDEF} of $\mathcal F_N$ and bounds~\eqref{E:BOUNDONE}--\eqref{E:BOUNDFOUR} we conclude that there exists a universal constant $\beta$ such that~\eqref{E:FNBOUND} holds.

We now turn our attention to the source term $\mathcal G_N$. By estimates analogous to~\eqref{E:BOUNDONE}--\eqref{E:BOUNDFOUR} we conclude that the absolute value of the first two lines of~\eqref{E:GNDEF} is bounded by 
\[
\frac{\beta}{\omega_0^2} \frac{C^{N-2}}{N}.
\]
Clearly 
\begin{align}
\lv  (-1)^N - 3 \sum_{k+m=N\atop k<N}\omega_k(-1)^m \rv
& \lesssim 1 + \sum_{k=0}^{N-1} \lv\omega_k\rv 
 \lesssim 1+ \sum_{k=1}^{N-1} \frac{C^{k-1}}{k^2} \notag \\
& = 1 + \sum_{k=1}^{\lfloor \frac N2\rfloor} \frac{C^k}{k^2}+\sum_{k=\lfloor \frac N2\rfloor+1}^{N-1} \frac{C^{k-1}}{k^2} \notag \\
& \lesssim 1 + \frac {C^{N-2}}{N} + C^{\lfloor \frac N2\rfloor} 
 \lesssim \frac{C^{N-2}}{N}, \label{E:THIRDLINEBOUND}
\end{align}
where we note that the last estimate follows from $N< C^{N-2-\lfloor \frac N2\rfloor}$, $N\ge 5$ and $C$ sufficiently large, but independent of $N$. To bound the quadratic nonlinearities in the 4-th line of~\eqref{E:GNDEF} we note the bound
\begin{align}
\lv \sum_{\ell+m=N \atop \ell<N}(-1)^m(\omega^2)_\ell  \rv & \le 1+\sum_{\ell=1}^{N-1} \lv (\omega^2)_\ell \rv 
\le 1+ \sum_{\ell=1}^{N-1} \sum_{k=0}^\ell |\omega_k||\omega_{\ell-k}| \notag \\
& \le 1+ \sum_{\ell=1}^{N-1} \left( |\omega_0| \frac{C^{\ell-1}}{\ell^2} + \sum_{k=1}^\ell \frac{C^{\ell-2}}{k^2(\ell-k)^2}\right) \notag \\
& \lesssim  \frac{C^{N-2}}{N}, \notag 
\end{align} 
where we have used the same argument as in the proof of~\eqref{E:THIRDLINEBOUND} to estimate 
$\sum_{\ell=1}^{N-1}  \frac{C^{\ell-1}}{\ell^2}$ and Lemma~\ref{L:COMB} to estimate $ \sum_{k=1}^\ell \frac{1}{k^2(\ell-k)^2}$. We estimate the remaining terms in the 4-th line of~\eqref{E:GNDEF} analogously and conclude
\begin{align}
\lv \sum_{\ell+m=N-2}(-1)^m(\omega^2)_\ell+2\sum_{\ell+m=N-1}(-1)^m(\omega^2)_\ell
 + \sum_{\ell+m=N \atop \ell<N}(-1)^m(\omega^2)_\ell + \sum_{k+n=N\atop 0<k<N}\omega_k\omega_n  \rv
\lesssim  \frac{C^{N-2}}{N}. \notag 
\end{align}
The cubic expressions in the last two lines of~\eqref{E:GNDEF} are estimates by the same token as the bound~\eqref{E:BOUNDTHREE} and we have
\begin{align}
& \lv 3\S^2 \Big(\sum_{k+\ell+m=N-2}(-1)^m\omega_k(\omega^2)_\ell
+ 2 \sum_{k+\ell+m=N-1}(-1)^m\omega_k(\omega^2)_\ell \rv \notag \\
& +\lv \sum_{k+\ell+m=N \atop k\neq N, \ell\neq N}(-1)^m\omega_k(\omega^2)_\ell
+ \sum_{k+\ell+m=N}(-1)^m\omega_0\sum_{a+b=N\atop 0<a<N }\omega_a\omega_b\Big)\rv \lesssim \frac{C^{N-3}}{N}.
\notag 
\end{align}
From the definition~\eqref{E:GNDEF} of $\mathcal G_N$ and the above bounds, we conclude~\eqref{E:GNBOUND}. 
\end{proof}


\begin{lemma}\label{L:INDUCTIVESTEP}
Let $\S>\frac32$.
Under the inductive assumptions~\eqref{E:ASS1}--\eqref{E:ASS2} it follows that for a sufficiently large $C$
\begin{align}
\lv \rho_N\rv & \le \frac{C^{N-1}}{N^2} \label{E:ISTEPRHO}\\
\lv \omega_N\rv & \le \frac{C^{N-1}}{N^2}.\label{E:ISTEPOMEGA}
\end{align}
Moreover, for any closed interval $K\subset (0,\frac23)$ we can choose the same constant $C$ for all $\omega_0\in K$. 
\end{lemma}


\begin{proof}
When $N=1$ the claim is obvious from~\eqref{E:LPCONDITION}. 
Bounds~\eqref{E:RHONBOUNDONE} and~\eqref{E:FNBOUND} together give
\begin{align}
\lv \rho_N\rv \le \frac{\alpha \beta}{\omega_0^2\left(\frac23-\omega_0\right)} \frac{C^{N-2}}{N^2} = \frac{\alpha\beta}{C\omega_0^2(\frac23-\omega_0)}  \frac{C^{N-1}}{N^2}. \notag 
\end{align}
Similarly, bounds~\eqref{E:OMEGANBOUNDONE} and~\eqref{E:GNBOUND} give
\begin{align}
\lv \omega_N\rv & \le  \frac{\alpha}{N\left(\frac23-\omega_0\right)} 
\left(\lv \mathcal G_N\rv + \frac1N  \lv \mathcal F_N\rv \right)
 \le \frac{2\alpha\beta}{C\omega_0^2\left(\frac23-\omega_0\right)}   \frac{C^{N-1}}{N^2}. \notag 
\end{align}
It is now clear that we can choose $C=C(\alpha,\beta,\omega_0)$ large enough so that~\eqref{E:ISTEPRHO}--\eqref{E:ISTEPOMEGA} is true. The uniformity statement with respect to a closed subinterval $K\subset (0,\frac23)$ is obvious from the previous estimate.
 \end{proof}


\begin{theorem}\label{T:ANALYTICITY}
Let $K\subset (0,\frac23)$ be a closed interval. There exists an $1>r=r_K>0$ such that for any $\omega_0=\frac1\S\in K$ the formal power series 
\begin{align}
{\bf \rho}(z) : = \sum_{N=0}^\infty  \rho_N  \dz^N , \ \ 
{\bf \omega}(z) : = \sum_{N=0}^\infty  \rho_N  \dz^N \label{E:INFINITESUMS}
\end{align}
converge for all $z$ such that $\lv \delta z\rv=\lv z-1\rv<r$. 
In particular, functions $\rho(z)$ and $\omega(z)$ are real analytic inside the ball $|z-1|<r$. We can differentiate the infinite sums term by term and $(\rho(z),\omega(z))$ is an LP-type solution of~\eqref{E:RHOEQN}--\eqref{E:OMEGAEQN} for $\lv z-1\rv<r$. Moreover, the density $\rho(\cdot;\S)$ is strictly positive  for $\lv z-1\rv<r$. 
\end{theorem}


\begin{proof}
By Lemma~\ref{L:INDUCTIVESTEP} there exists a $C=C_K>0$ such that 
\begin{align}
\sum_{N=1}^\infty  \lv \rho_N\rv  |\delta z|^N + \sum_{N=1}^\infty  \lv \omega_N\rv  |\delta z|^N \le  2\sum_{N=1}^\infty  \frac{\lv C\delta z\rv^N}{CN^2} <\infty, \notag 
\end{align}
when $|\delta z|<\frac1C=:r$. The claim follows by the comparison test. The real analyticity and differentiability statements are clear. Since
\begin{align}
1-y_\ast^2z^2\omega^2 & = 1-y_\ast^2(1+\delta z)^2\omega^2 \notag \\
& = - 2y_\ast(1-\omega_0) \delta z - y_\ast^2 \sum_{N=2}^\infty \left((\omega^2)_{N-2}+2(\omega^2)_{N-1}+(\omega^2)_N\right)\dz^N \notag \\
& \neq0, \ \ 0<|\delta z| \ll1, \notag 
\end{align}
it follows that for $r>0$ sufficiently small, the function $1-y_\ast^2z^2\omega^2\neq 0$ for all $|z-1|<r$ and $z\neq1$.
Functions $\rho(z)$ and $\omega(z)$ are indeed the solutions, as can be seen by plugging the infinite series~\eqref{E:INFINITESUMS} into the left-hand sides of~\eqref{E:OMEGAEQN1}--\eqref{E:RHOEQN1}; all the functions appearing on the left-hand side are analytic for $|z-1|<r$. 
\end{proof}


We note that 
$\rho_i,\omega_i$ are smooth with respect to $\omega_0$ (or equivalently $y_\ast$) for $i=0,1$ for any $\omega_\ast>0$. This follows from the explicit formula for the Taylor coefficients around the sonic point for LP-type solutions. We next note that for any $N\ge 2$ we can express $\rho_N, \omega_N$ as polynomial function of $\rho_0,\dots\rho_{N-1}, \omega_0,\dots\omega_{N-1}$ for any $0<\omega_0<\frac23$ and therefore it is clear that $\rho_N,\omega_N$ are smooth functions of $\omega_0$  for all $N\in\mathbb N$ for $\omega_0\in(0,\frac23)$.


\begin{lemma}\label{L:ANALYTIC}
For any $\omega_0\in(0,\frac23)$ there exists a constant $C =C(\omega_0)>0$ such that for all $N\in\mathbb N\setminus \{0\}$:
\begin{align}
\lv \pa_{\omega_0}\rho_N\rv & \le \frac{C^{N}}{N^2} \label{E:OMEGAREG1}\\
\lv \pa_{\omega_0}\omega_N\rv & \le \frac{C^{N}}{N^2}  .\label{E:OMEGAREG2}
\end{align}
In particular, there exists an $r>0$ such that the formal power series
\begin{align}
 \sum_{N=0}^\infty \pa_{\omega_0}\rho_N \dz^N, \ \ 
 \sum_{N=0}^\infty \pa_{\omega_0}\omega_N \dz^N,  \notag 
\end{align}
converge for all $z$ satisfying $|z-1|<r$. Moreover, the function $(0,\frac23)\ni \omega_0\to (\rho(z;\omega_0),\omega(z,\omega_0)$
is $C^1$ and the derivatives $\pa_{\omega_0}\rho$ and  $\pa_{\omega_0}\omega$ are given by the infinite series above.
\end{lemma}


\begin{proof}
When $N=1$ we have by the Larson-Penston sonic condition $\pa_{\omega_0}\rho_1 = -1$, $\pa_{\omega_0}\omega_1=-2$ and the claim is obvious. 
We now note that by~\eqref{E:RHONFN}--\eqref{E:OMEGANFNGN} 
\begin{align}
\pa_{\omega_0}\rho_N  = & - \frac{N}{\omega_0^2\left(N(1-\frac1{\omega_0})+1\right)} \rho_N 
+ \frac1{2\left(N(1-\frac1{\omega_0})+1\right)} \pa_{\omega_0}\mathcal F_N \label{E:PARTIALOMEGARHO}\\
\pa_{\omega_0}\omega_N & = - \frac1{2N\omega_0^2(1-\frac1{\omega_0})^2} \mathcal G_N 
- \frac{2N\left(1-\frac1{\omega_0}\right)+1}{2N\omega_0^2\left(1-\frac1{\omega_0}\right)^2\left (N(1-\frac1{\omega_0})+1\right)^2}\mathcal F_N \notag \\
& \ \  + \frac1{2N(1-\frac1{\omega_0})} \pa_{\omega_0}\mathcal G_N
+ \frac1{2N(1-\frac1{\omega_0})\left(N(1-\frac1{\omega_0})+1\right)} \pa_{\omega_0}\mathcal F_N. \label{E:PARTIALOMEGAOMEGA}
\end{align} 
From~\eqref{E:PARTIALOMEGARHO} and Lemma~\ref{L:KEYINDUCTION} we immediately have the bound
\begin{align}
\lv \pa_{\omega_0}\rho_N \rv & \lesssim \frac 1 {\omega_0^2(\frac23-\omega_0)} |\rho_N| + \frac{1}{N\left(\frac23-\omega_0\right)}\lv\pa_{\omega_0}\mathcal F_N\rv \notag \\
& \le \frac{C^{N-1}} {N^2\omega_0^2(\frac23-\omega_0)} + \frac{1}{N\left(\frac23-\omega_0\right)}\lv\pa_{\omega_0}\mathcal F_N\rv \notag 
\end{align}
Similarly, from~\eqref{E:PARTIALOMEGAOMEGA} and Lemma~\ref{L:KEYINDUCTION} we obtain 
\begin{align}
\lv \pa_{\omega_0}\omega_N \rv & \lesssim \frac1N \lv \mathcal G_N\rv + \frac1{N^2(\frac23-\omega_0)}\lv \mathcal F_N\rv
+ \frac1{N^2(\frac23-\omega_0)}\lv \pa_{\omega_0}\mathcal F_N\rv +  \frac1{N(\frac23-\omega_0)}\lv \pa_{\omega_0}\mathcal G_N\rv \notag \\
& \le \frac{\beta}{\omega_0^2} \frac{C^{N-2}}{N^2} + \frac1{(\frac23-\omega_0)}\frac{\beta}{\omega_0^2} \frac{C^{N-2}}{N^3}
+ \frac1{N^2(\frac23-\omega_0)}\lv \pa_{\omega_0}\mathcal F_N\rv +  \frac1{N(\frac23-\omega_0)}\lv \pa_{\omega_0}\mathcal G_N\rv \notag 
\end{align}

We now recall that $\mathcal F_N$ and $\mathcal G_N$ are cubic polynomials in $2N$ variables $\rho_0,\omega_0,\dots,\rho_{N-1},\omega_{N-1}$. When differentiating with respect to $\omega_0$ at most one term, indexed by $\rho_k$ or $\omega_k$, $0\le k\le N-1$ is differentiated. In particular, the same combinatorial structure in the problem is maintained and the same inductive proof relying on the already established bounds~\eqref{E:ISTEPRHO}--\eqref{E:ISTEPOMEGA} gives~\eqref{E:OMEGAREG1}--\eqref{E:OMEGAREG2}. The remaining conclusions now follow easily.
\end{proof}


\subsection{Existence, uniqueness, and regularity near the origin}

\begin{theorem}\label{T:ANALYTICITYLEFT}
Let $\rho_0>0$ be given. There exists an $0<\tilde r<1$ such that the formal power series 
\begin{align}
\rho_-(z,\rho_0):=\sum_{N=0}^\infty \tr_N z^N, \ \
\omega_-(z;\rho_0):= \sum_{N=0}^\infty \tom_N z^N \notag 
\end{align}
converge for all $z$ such that $0\le  z<\tilde r$. 
In particular, functions $\rho_-(\cdot;\rho_0)$ and $\omega_-(\cdot;\rho_0)$ are real analytic on $[0,\tilde r)$. We can differentiate the infinite sums term by term and the functions $\rho_-(\cdot;\rho_0)$ and $\omega_-(\cdot;\rho_0)$  solve~\eqref{E:RHOEQN}--\eqref{E:OMEGAEQN} with the initial conditions $\omega_-(0;\rho_0)=\frac13$ and $\rho_-(0;\rho_0)=\rho_0$.
\end{theorem}

\begin{proof}
By analogy to the previous section, we must Taylor-expand the solution at the origin $z=0$ in order to prove a local existence theorem starting from the left. An immediate consistency condition follows from the presence of $\frac{1-3\omega}{z}$ on the right-hand side of~\eqref{E:OMEGAEQN}. Namely, in order to have a well-posed problem with initial data prescribed at $z=0$ we must have $\omega(0)=\frac13$. 

Assume that locally around $z=0$
\begin{align}
\rho & =  \sum_{N=0}^\infty \tr_N z^N, \ \ \omega  = \sum_{N=0}^\infty \tom_N z^N \label{E:RHOOMEGAEXPANSIONZERO} 
\end{align}

Our starting point are the equations
\begin{align}
\omega' (1-\S^2 z^2\omega^2) - \left(1-3\omega\right)(1-\S^2z^2\omega^2)\frac1z -
2\S^2z\omega^2(\rho-\omega) & =0, \label{E:OMEGAEQN1ZERO}\\
\rho' (1-\S^2z^2\omega^2) +2\S^2 z \omega\rho(\rho-\omega) &=0, \label{E:RHOEQN1ZERO}
\end{align}

We plug in~\eqref{E:RHOOMEGAEXPANSIONZERO} into~\eqref{E:RHOEQN1ZERO} 
and obtain
\begin{align}
0 & = \left(\sum_{k=0}^\infty k \tr_k z^{k-1}\right)\left(1 - \S^2\sum_{\ell=0}^\infty(\omega^2)_\ell z^{\ell+2}\right) 
 + 2 \S^2 \sum_{k=0}^\infty \left(\omega\rho(\rho-\omega)\right)_k z^{k+1} \notag \\
& = \sum_{N=0}^\infty (N+1)\tr_{N+1}z^N 
 - \S^2 \sum_{N=0}^\infty  \sum_{k+\ell = N-2}(k+1)\tr_{k+1}(\omega^2)_\ell z^N + 2 \S^2 \sum_{N=0}^\infty \left(\omega\rho(\rho-\omega)\right)_{N-1}  z^N, \notag 
\end{align}
where, by definition $\tr_k=\tom_k=0$ for $k<0$. Equating the coefficients above, we conclude that for any non-negative $N$ we have
\begin{align}
& (N+1)\tr_{N+1} - \S^2 \sum_{k+\ell = N-2}(k+1)\rho_{k+1}(\omega^2)_\ell 
+ 2 \S^2 \left(\omega\rho(\rho-\omega)\right)_{N-1} = 0, \label{E:TAYLORRHO1ZERO}
\end{align}
which is precisely~\eqref{E:TAYLORRHOFORMULA}. Similarly, after plugging in~\eqref{E:RHOOMEGAEXPANSIONZERO} into~\eqref{E:OMEGAEQN1ZERO} we obtain
\begin{align}
0 & = \left(\sum_{k=0}^\infty k \tom_k z^{k-1}\right)\left(1 - \S^2\sum_{\ell=0}^\infty(\omega^2)_\ell z^{\ell+2}\right) 
+ 3 \sum_{k=0}^\infty\tom_{k+1}z^k\left(1 - \S^2\sum_{\ell=0}^\infty(\omega^2)_\ell z^{\ell+2}\right)  \notag \\
& \ \ \ \ - 2 \S^2 \sum_{k=0}^\infty \left(\omega^2(\rho-\omega)\right)_k z^{k+1} \notag \\
& = \sum_{N=0}^\infty (N+1)\tom_{N+1}z^N 
- \S^2 \sum_{N=0}^\infty  \sum_{k+\ell = N-2}(k+1)\tom_{k+1}(\omega^2)_\ell z^N 
+ 3 \sum_{N=0}^\infty  \tom_{N+1}z^N \notag \\
& \ \ \ \ - 3 \S^2 \sum_{N=0}^\infty \sum_{k+\ell = N-2} \tom_{k+1}(\omega^2)_\ell z^N
- 2 \S^2 \sum_{N=0}^\infty \left(\omega^2(\rho-\omega)\right)_{N-1}  z^N. \notag 
\end{align}
Therefore for any $N\ge0$ we have
\begin{align}
(N+4)\tom_{N+1} - \S^2 \sum_{k+\ell = N-2}(k+1)\tom_{k+1}(\omega^2)_\ell 
- 3 \S^2\sum_{k+\ell = N-2} \tom_{k+1}(\omega^2)_\ell
- 2 \S^2 \left(\omega^2(\rho-\omega)\right)_{N-1}  =0 \label{E:TAYLOROMEGA1ZERO}
\end{align}

It is clear from~\eqref{E:TAYLORRHO1ZERO} that $\rho_0=\rho(0)$ is a free parameter. Identities~\eqref{E:TAYLORRHO1ZERO}--\eqref{E:TAYLOROMEGA1ZERO} give the recursive relationships
\begin{align}
\tr_{N+1} & = \frac1{N+1} \tilde{\mathcal F}_{N+1}, \ \ N\ge0 \label{E:RHONATZERO}\\
\tom_{N+1} & = \frac1{N+4} \tilde{\mathcal G}_{N+1}, \ \ N\ge0. \label{E:OMEGANATZERO}
\end{align}
where 
\begin{align}
\tilde{\mathcal F}_{N+1} & : = \S^2 \sum_{k+\ell = N-2}(k+1)\tr_{k+1}(\omega^2)_\ell 
- 2 \S^2 \left(\omega\rho(\rho-\omega)\right)_{N-1}  \label{E:TILDEFNDEF}\\
\tilde{\mathcal G}_{N+1} & : =\S^2 \sum_{k+\ell = N-2}(k+1)\tom_{k+1}(\omega^2)_\ell 
+3 \S^2\sum_{k+\ell = N-2} \tom_{k+1}(\omega^2)_\ell
- 2 \S^2 \left(\omega^2(\rho-\omega)\right)_{N-1} \label{E:TILDEGNDEF}
\end{align}
The rest of the proof is now entirely analogous to the proof of Theorem~\ref{T:ANALYTICITY} and we leave out the details.
\end{proof}

\begin{remark}
Letting $N=0$ in~\eqref{E:RHONATZERO}--\eqref{E:OMEGANATZERO} we immediately see from~\eqref{E:TILDEFNDEF}--\eqref{E:TILDEGNDEF} that $\tilde{\mathcal F}_{1}=\tilde{\mathcal G}_{1}=0$ and therefore $\tr_1=\tom_1=0$.
Letting $N=1$ in~\eqref{E:TAYLORRHO1ZERO}--\eqref{E:TAYLOROMEGA1ZERO} we obtain
\begin{align}
\tr_2 &= - \S^2 \frac13\rho_0(\rho_0-\frac13) = - \frac13 \S^2 \rho_0^2 + \frac19 \S^2 \rho_0  \notag \\
\tom_2 & = \frac 2{45} \S^2(\rho_0-\frac13) = - \frac{2\S^2}{135}+   \frac{2\S^2}{45} \rho_0 \label{E:BCATZEROOMEGA}
\end{align}
By~\eqref{E:BCATZEROOMEGA} we have in the vicinity of $z=0$ 
\[
\pa_{\rho_0} \omega_-(z;\rho_0) =  \frac{2\S^2}{45} z^2 + \sum_{N=3}^\infty \tom_Nz^N >0, \ \ 0<z<\tilde r.
\]
\end{remark}

\section{The outer region $z>1$}\label{S:OUTERREGION}

In this section we show that for any value of $\S\in[2,3]$ there exists a unique LP-type solution in the exterior region. Such a statement is true because the flow ``moves" in a stable direction as $z\to\infty$. This should be contrasted to the more delicate analysis of the flow in the inner region. Our first preparatory lemma lists a number of simple properties in the vicinity of $z=1$, which follow by continuity and careful use of the LP condition~\eqref{E:LPCONDITION}.

\begin{lemma}[Initialisation]\label{L:INITIAL}
Let $y_\ast\in[2,3]$
and let $(\rho(\cdot),\omega(\cdot)):=(\rho(\cdot;y_\ast),\omega(\cdot;y_\ast))$ be the unique local LP-type solution to~\eqref{E:RHOEQN}--\eqref{E:OMEGAEQN} given by Theorem~\ref{T:ANALYTICITY}. Then there exists a $\delta>0$ such that the following bounds hold:
\begin{align}
\rho'(z)&<0, \ \ z\in (1,1+\delta), \label{E:C1}\\
\omega'(z) & >0, \ \ z\in (1, 1+\delta) \label{E:C2}\\
\frac13< \omega(z) &< 1, \ \ z\in (1, 1+\delta) \label{E:C3}\\
-2\frac{\rho(z)}{z}<\rho'(z)&<-\frac{\rho(z)}{ z}, \ \ z\in (1, 1+\delta) \label{E:C3.5}\\
\omega(z) & >\frac1{y_\ast z}, \ \ z\in (1, 1+\delta) \label{E:C4}\\
\frac1{y_\ast z} & > \rho(z), \ \ \ \ z\in (1, 1+\delta)  \label{E:C5}\\
\rho(z) \omega(z) & >(y_\ast z)^{-2}, \ \ z\in (1, 1+\delta) \label{E:C6} 
\end{align}
\end{lemma}

\begin{proof}
Claims~\eqref{E:C1}--\eqref{E:C3} are clear and follow by a continuity argument from~\eqref{E:LPCONDITION}. To prove claim~\eqref{E:C3.5} 
we first note that due to $\rho'(1)=-\frac1{y_\ast}$ we have $\rho'(z)+\frac{\rho(z)}{z}=0$ at $z=1$. 
Notice that for any solutions of LP-type by~\eqref{E:LPCONDITION}
\be\label{E:S2NEGATIVE}
2\rho_2\S-2 = - \frac{y_\ast^2-2y_\ast+1}{(2y_\ast-3)} <0,
\ee
where we recall $\rho_2=\frac12\rho''(1)$ by~\eqref{E:RHOOMEGAEXPANSION}.
Since $\frac{d}{dz}\left(\rho'(z)+\frac{\rho(z)}{z}\right)\Big|_{z=1}= \rho''(z)-\frac{\rho(z)}{z^2}+\frac{\rho'(z)}{z}\Big|_{z=1}
=2\rho_2 - 2 \frac1{y_\ast}<0$ by~\eqref{E:S2NEGATIVE}, claim follows by a continuity argument.
Claim~\eqref{E:C4} follows since $\omega(1)=\frac1{y_\ast}$ and $\omega$ is by~\eqref{E:C2} locally strictly increasing and
$\frac1{y_\ast z}$ is clearly strictly decreasing. To prove~\eqref{E:C5}, consider
\[
F(z): = \frac1{y_\ast z} - \rho(z).
\]
Note that $F(1)=F'(1)=0$ and it is therefore necessary to evaluate the second derivative of $F$. A direct calculation   shows that 
$F''(1) = \frac2{y_\ast}-2\rho_2$ which is strictly positive by ~\eqref{E:S2NEGATIVE}. Thus $F$ is strictly increasing on $(1, 1+\delta)$ for a sufficiently small $\delta$.

Finally,  from~\eqref{E:RHOEQN}--\eqref{E:OMEGAEQN} we obtain the equation
\begin{align}\label{E:USEFUL}
\left(\rho \omega z^2\right)' = (1-\omega) \rho  z.
\end{align}
By~\eqref{E:C3} we conclude that $\rho \omega z^2$ is strictly increasing on the interval $(1-\delta, 1+\delta)$ and therefore, since 
$\rho\omega=y_\ast^{-2}$ at $z=1$, claim~\eqref{E:C6} follows.
\end{proof}

 
The next lemma shows the crucial dynamic trapping property. 
\begin{lemma}[Invariant set]\label{L:INVARIANT}
Let $y_\ast\in[2,3]$ and
let $(\rho(\cdot;y_\ast),\omega(\cdot;y_\ast))$ be the unique local LP-type solution to~\eqref{E:RHOEQN}--\eqref{E:OMEGAEQN} given by Theorem~\ref{T:ANALYTICITY}. Let $I = (1, T)$ be the maximal interval of existence to the right of $z=1$ on which the properties
\begin{align}
\frac1{y_\ast z} & > \rho(z),   \label{E:C5NEW}\\
\rho(z)\omega(z) & >(y_\ast z)^{-2},  \label{E:C6NEW} 
\end{align}
hold. Note that $T\ge 1+\delta>1$ by Lemma~\ref{L:INITIAL}. Then the following bounds hold:
\begin{align}
\frac13&< \omega(z) < 1, \ \ z\in (1, T) \label{E:C3NEW}\\
-2\frac{\rho(z)}{ z}&<\rho'(z)<-\frac{\rho(z)}{ z}, \ \ z\in (1, T) \label{E:C3.5NEW}
\end{align}
\end{lemma}

\begin{proof}
By~\eqref{E:C5NEW}--\eqref{E:C6NEW} we have $\omega>\rho$ on $(1, T)$ and 
\be
\omega(z)  >\frac1{y_\ast z} \ \ z\in (1, T). \label{E:C4NEW}
\ee
Therefore from~\eqref{E:RHOEQN} $\rho'<0$ on
$(1,T)$. 

\noindent
{\bf Proof of~\eqref{E:C3NEW}.}
We note that on $(1,T)$ due to $\omega>\rho$ and $\omega(z) >\frac1{y_\ast z}$ we have from~\eqref{E:OMEGAEQN}
$
\omega'(z)>\frac{1-3\omega(z)}{z}.
$
Integrating over $[1, z]$ for any $z\in(1,T)$ we conclude
\[
\omega(z) > (\frac1{y_\ast}-\frac13) z^{-3} + \frac13 \ge \frac13, \ \ z\in (1,T).
\]

We may rewrite~\eqref{E:OMEGAEQN} in the form
\begin{align}\label{E:REQN2}
\omega' = \frac{1-\omega}z + \frac{-2\omega+2y_\ast^2z^2\omega^2\rho}{z\left(1-y_\ast^2z^2\omega^2\right)}.
\end{align}
From~\eqref{E:C6NEW} we have on $(1,T)$ $-2\omega+2y_\ast^2z^2\omega^2\rho>0$ and therefore from~\eqref{E:REQN2} (together with $\omega(z) y_\ast z >1$) we conclude that 
$
\omega'(z)<\frac{1-\omega(z)}{z}.
$
Integrating over $[1, z]$ for any $z\in(1,T)$ we conclude from $y_\ast>1$
\[
\omega(z) < (\frac1{y_\ast}-1) z^{-1} + 1 < 1, \ \ z\in (1,T).
\]
Therefore~\eqref{E:C3NEW} holds.

\noindent
{\bf Proof of~\eqref{E:C3.5NEW}.}
We may rewrite~\eqref{E:RHOEQN} in the form
\begin{align}\label{E:ASIMP}
\frac{\rho'z}{\rho} = -2 +2 \frac{1-y_\ast^2 z^2\omega  \rho}{1-y_\ast^2 z^2\omega^2}.
\end{align}
Since by~\eqref{E:C6NEW} $1-\S^2 z^2 \omega \rho<0$ the lower bound follows immediately.
To prove the upper bound we rewrite~\eqref{E:RHOEQN} in the form
\begin{align}
\frac{\rho'z}{\rho} = -1 + \frac{1+y_\ast^2z^2 \omega^2-2y_\ast^2z^2 \omega \rho}{1-y_\ast^2z^2 \omega^2} = -1 + \frac{(1-y_\ast z\omega)^2+2y_\ast z \omega(1-y_\ast z \rho)}{1-y_\ast^2 z^2 \omega^2}. \label{E:EXTRA}
\end{align}
The last expression is strictly less than $-1$ by~\eqref{E:C4NEW} and~\eqref{E:C5NEW}.
\end{proof}


Finally, combining the previous two lemmas we obtain the desired forward global existence result in the outer region $z\ge1$.

\begin{proposition}[Forward global existence]\label{P:FGE}
Let $y_\ast\in[2,3]$ and
let $(\rho(\cdot;y_\ast),\omega(\cdot;y_\ast))$ be the unique local LP-type solution to~\eqref{E:RHOEQN}--\eqref{E:OMEGAEQN} given by Theorem~\ref{T:ANALYTICITY}.  Then there exists a unique forward global solution to~\eqref{E:RHOEQN}--\eqref{E:OMEGAEQN} on $[1,\infty)$ satisfying the following properties:
\begin{align}
\frac13 & <\omega(z)<1 \label{E:OMEGAUPDOWN}\\
-2& < \frac{\rho'(z) z}{\rho(z)} <-1, 
\end{align}
Moreover,
\begin{align}\label{E:ASBE}
\rho(z)= \frac{C}{z^2} \left(1+ O_{z\to\infty}(\frac{1}{z})\right), \ \ \omega(z)=1+ O_{z\to\infty}\left(\frac1z\right)
\end{align}
\end{proposition}

\begin{proof}
Let $T$ be defined as in Lemma~\ref{L:INVARIANT} and assume 
that $T<\infty$. Notice that by the bounds in Lemma~\ref{L:INVARIANT}, both $\omega$ and $\rho$ remain bounded and away from the sonic point singularity for $z\in (1,T)$. 
At $T$ we must have either $\frac1{y_\ast T} =\rho(T)$ or $\rho(T)\omega(T)  = \frac1{y_\ast^2 T^2}$. 

Let $\frac1{y_\ast T} =\rho(T)$. Since $\left(\rho y_\ast z\right)'=\rho\S(1+\frac{\rho'z}{\rho})<0$ by~\eqref{E:C3.5NEW} for all $z\in(1,T)$, 
we must have $\rho(T)y_\ast T<\rho(1) y_\ast = 1$, a contradiction.

Let now $\rho(T)\omega(T)  = \frac1{y_\ast^2 T^2}$. Since $\omega<1$ on $(1,T)$ we conclude from~\eqref{E:USEFUL}
that $z^2 \rho(z)\omega(z)$ is strictly increasing on $(1,T)$. Therefore
\begin{align}
T^2\rho \omega > y_\ast^2 \rho(1)\omega(1) =1
\end{align}
a contradiction.
Therefore, the solution $(\rho(\cdot;y_\ast),\omega(\cdot;y_\ast))$ exists for all $z>1$.
Finally, since $\omega(z)>\frac13$ on $(1,\infty)$ and $z\rho\le\frac1{\S}$ by the above bounds, we conclude easily that $\lv\frac{1-y_\ast^2 z^2\omega  \rho}{1-y_\ast^2 z^2\omega^2}\rv\lesssim \frac1z$, $z>1$. 
It follows from~\eqref{E:ASIMP} that $\frac{\rho'z}{\rho} = -2 + O(\frac{1}{z})$ and this implies the $\rho$-asymptotics in~\eqref{E:ASBE}. From~\eqref{E:RHOEQN}--\eqref{E:OMEGAEQN} it is easy to see that $(\omega z)' = 1-\frac{\rho'}{\rho}z\omega -2\omega= 1+ \omega O(\frac{1}{z^2})$, where in the last equality we have used the $\rho$-asymptotics~\eqref{E:ASBE} and~\eqref{E:ASIMP}. This easily gives the $\omega$-asymptotics in~\eqref{E:ASIMP}.
%
\end{proof}

\section{The inner region $0\le z<1$}\label{S:INNERREGION}

By Theorem~\ref{T:ANALYTICITY} and Lemma~\ref{L:ANALYTIC} there exists an $r>0$ such that for any $\S\in[2,3]$ there is unique LP-type solution $(\rho(z;y_\ast),\omega(z;y_\ast))$ on $(1-r,1+r)$, which is analytic-in-$z$ and uniformly continuous with respect to $y_\ast$. The next lemma records the obvious statement that one can extend the existence interval as long as we are away from the sonic line. 

\begin{lemma}[Local existence and uniqueness away from the sonic line]\label{L:LOCALEX}
Let $\S\in[2,3]$ be given and assume that for some $z\in(0,1)$ the conditions
\begin{align}
1 - z^2\S^2\omega(z)^2>0, \ \ \rho(z)>0 \notag 
\end{align}
hold. Then there exists a unique smooth local-in-$z$ solution to the initial value problem~\eqref{E:RHOEQN}--\eqref{E:OMEGAEQN} on some time interval $(z-T,z+T)\subset(0,1)$. 
\end{lemma}

\begin{proof}
The proof is a standard consequence of the local well-posedness theory for ordinary differential equations.
\end{proof}


To every $y_\ast\in[2,3]$ (i.e. $\omega_0=\omega(1)\in[\frac13,\frac12]$) we associate
the {\em sonic} time of existence to the left
\begin{align}\label{E:SUBSONIC}
s(y_\ast): = \inf \left\{z\in(0,1)\,\big| \text{ solution exists on $(z,1]$ and }\ \ \omega^2(z;y_\ast)z^2y_\ast^2<1\right\}.
\end{align} 
Clearly $0\le s(y_\ast)< 1-r$.
By Lemma~\ref{L:LOCALEX} the solution can be continued to the left starting at $z=1-r$ for a short time and the maximal time of existence to the left is smaller or equal to $s(y_\ast)$. Sonic time is of central importance in our analysis and our first goal is to show that there exists a $\S\in[2,3]$ such that $s(\S)=0$. To that end we prove a number of preparatory lemmas. We next collect important a priori bounds that hold on $(s(\S),1)$. 




\begin{lemma}\label{L:APRIORIODE}
Let $\S\in[2,3]$ be given and consider the unique LP-type solution $(\rho(\cdot;y_\ast),\omega(\cdot;y_\ast))$ to the left of $z=1$.
For any $z\in (s(y_\ast),1)$ we have the a priori bounds
\begin{align}
0 < \rho(z) & <\frac1{\S z}, \label{E:RHOAPRIORI}\\  
|\omega(z)| & <\frac1{\S z}, \label{E:OMEGAPRIORI} \\
(z\rho(z))' & >0. \label{E:ZRHOMONOTONE}
\end{align}
\end{lemma}


\begin{proof}
Let $\mathring z\in (s(y_\ast),1)$ be given. 
From the definition~\eqref{E:SUBSONIC} it follows that there exists an $\eta>0$ such that $\omega(z)^2z^2y_\ast^2<1-\eta$ for all $z\in[\mathring z,1-r]$ and in particular
\begin{align}\label{E:CETA}
\omega(z)^2 \le \frac{1-\eta}{y_\ast^2\mathring z^2}=:C_\eta, \ \ z\in[\mathring z,1-r].
\end{align}

We first show that $\rho$ remains positive on $(s(y_\ast),1)$.
Let $\bar z:=\inf_{z\in[\mathring z,1]}\left\{\rho(\zeta)>0 \ \text{for all} \ \zeta\in(z,1]\right\}$.  By Theorem~\ref{T:ANALYTICITY} we have $\bar z < 1-r$.
Since $\rho>0$ on $(\bar z,1]$ equation~\eqref{E:RHOEQN} gives
\be
\left(\log\rho\right)' =   -\frac{2\S^2 z \omega}{1-\S^2 z^2\omega^2}(\rho-\omega), \ \ z\in(\bar z,1]. \notag 
\ee
Therefore for any $z\in(\bar z,1-r]$ we get
\begin{align}
\rho(z) & =\rho(1-r) \exp\left(\int_{z}^{1-r} \frac{2\S^2 \zeta \omega}{1-\S^2 \zeta^2\omega^2}(\rho-\omega)\,d\zeta\right)  \notag \\
& \ge \rho(1-r) \exp\left(- \frac{2 \S^2 C_\eta^2}{\eta}\right) \exp\left(- \frac{2 \S^2 C_\eta}{\eta} \int_z^{1-r} \rho(\zeta)\,d\zeta \right). 
\notag 
\end{align}
The right-hand side is strictly positive and as $z\to\bar z$ it clearly remains strictly positive. Therefore $\bar z = \mathring z$.

In order to prove the upper bound for $\rho$, 
we consider  
\begin{align}
f(z) : = 1 - \S z \rho(z). \notag 
\end{align} 
Using~\eqref{E:RHOEQN} it is checked that 
\begin{align}\label{E:FEQN}
f'(z)  + f(z) \frac{2z\S^2\omega\rho}{1-\S^2z^2\omega^2} = - \frac{\S (1-\S z \omega)^2 \rho}{1-\S^2z^2\omega^2}
\end{align}
By the LP-type sonic conditions~~\eqref{E:ZEROCONST} and~\eqref{E:LPCONDITION} it is easy to see that $f(1)=f'(1)=0$. To determine the sign of $f$ near $z=1$ it is therefore necessary to compute the second derivative of $f$. 
Since
\begin{align}\notag
f''(z)\big|_{z=1} = -2\S\rho'(1) - \S \rho''(1) = 2-\frac{-y_\ast^2+6y_\ast-7}{ 2y_\ast-3} = \frac{y_\ast^2-2y_\ast+1}{2y_\ast-3},
\end{align}
we conclude that $f>0$ locally around $1$ as the above expression is strictly positive for $y_\ast\in[2,3]$. In fact, by choosing a possibly smaller $r$, we may assume $f(z)>0$ for all $z\in[1-r,1)$. 
We let $\tilde z:=\inf_{z\in[\mathring z,1]}\left\{f(\zeta)>0 \ \text{for all} \ \zeta\in(z,1]\right\}$. Since $\rho>0$ on $[\mathring z,1]$ it follows that the right-hand side of~\eqref{E:FEQN} is negative  for any $z\in(\tilde z, 1]$. Integrating~\eqref{E:FEQN} for any $z\in [\tilde z, 1-r]$ we get
\begin{align}
f(z) \ge f(1-r) \exp \left(\int_{z}^{1-r} \frac{2\zeta\S^2\omega\rho}{1-\S^2\zeta^2\omega^2}\,d\zeta\right).  \notag 
\end{align}
By an analogous argument as above, we conclude $f(\tilde z)>0$ and therefore $\tilde z = \mathring z$. Therefore $f(z)>0$ on $(s(y_\ast),1)$ as claimed. From~\eqref{E:FEQN} we then conclude $f'(z)<0$ which is equivalent to~\eqref{E:ZRHOMONOTONE}. 
\end{proof}


The following lemma shows that solutions which are a finite distance $\eta$ away from the sonic line and defined for all $z\ge \bar z>0$, can be extended to the left by a finite time depending only on $\eta$ and $\bar z$.

\begin{lemma}\label{L:TQUANT}
Let $\S\in[2,3]$ be given and consider the unique LP-type solution $(\rho(\cdot;y_\ast),\omega(\cdot;y_\ast))$ to the left of $z=1$, given by Theorem~\ref{T:ANALYTICITY}.
Assume that for some $\bar z\in(0,1-r)$ and $\eta>0$ we have $\bar z>s(y_\ast)$ and the conditions
\begin{align}\label{E:ETAZERO}
1 - z^2\S^2\omega(z;y_\ast)^2>\eta, \ \ \rho(z)>0, \ \ z\in[\bar z,1-r],
\end{align}
hold. Then there exists a $ t=t(\eta,\bar z)>0$ such that the solution can be continued to the interval $[\bar z- t,1)$ so that 
\begin{align}
1 - z^2\S^2\omega(z;y_\ast)^2>0, \ \ \rho(z)>0, \ \ z\in[\bar z- t,1-r]. \notag 
\end{align}
\end{lemma}



\begin{proof}
By~Lemma~\ref{L:APRIORIODE}
the following a priori bounds hold:
\begin{align}
\|\omega\|_{C^0([\bar z,1-r])}& < \frac1{\S \bar z},  \label{E:APRIORIOMEGAZERO}\\
\|\rho\|_{C^0([\bar z,1-r])}& <  \frac1{\S \bar z}. \label{E:APRIORIRHOZERO}
\end{align}
Formally, for any $0<z\le \bar z$ we write the equations~\eqref{E:RHOEQN}--\eqref{E:OMEGAEQN} in their integral form
\begin{align}
\omega(z) & = \omega(\bar z) + \int_{z}^{\bar z} \frac{3\omega-1}{\tau}\,d\tau 
 - 2\S^2\int_{z}^{\bar z} \mathcal F(\S,\rho,\omega)(\tau)\, d\tau,  \label{E:OMEGAEQNINTEGRATED}\\
 \rho(z) & =\rho(\bar z)  + 2\S^2\int_{z}^{\bar z} \mathcal G(\S,\rho,\omega)(\tau)\, d\tau,
 \label{E:RHOEQNINTEGRATED}
\end{align}
where
\begin{align}
\mathcal F(\S,\rho,\omega)(z) & : = \frac{z \omega^2(\rho-\omega)}{1-\S^2 z^2\omega^2}, \label{E:FDEF}\\
\mathcal G(\S,\rho,\omega)(z)& : = \frac{z \omega\rho(\rho-\omega)}{1-\S^2 z^2\omega^2}. \label{E:GDEF}
\end{align}
This motivates the following Picard iteration, where we let
\begin{align}
 \rho_n(z) & =\rho(\bar z)  + 2\S^2\int_{z}^{\bar z} \mathcal G(\S,\rho_{n-1},\omega_{n-1})(\tau)\, d\tau.
 \label{E:RHON}\\
\omega_n(z) & = \omega(\bar z) + \int_{z}^{\bar z} \frac{3\omega_{n-1}-1}{\tau}\,d\tau 
 - 2\S^2\int_{z}^{\bar z} \mathcal F(\S,\rho_{n-1},\omega_{n-1})(\tau)\, d\tau.  \label{E:OMEGAN}
\end{align}
For an $M>1$ sufficiently large and $t = t(\bar z,\eta)<\frac{\bar z}2 $ sufficiently small (both to be specified below), we make the inductive assumptions
\begin{align}
\lv \omega_k(z) \rv & \le \frac 4{\S\bar z}, \ \ z\in [\bar z- t, \bar z], \ \ k=0,1,2,\dots, n-1, \label{E:OMEGAIND} \\
\lv \rho_k(z) \rv & \le M, \ \ z\in [\bar z-t, \bar z], \ \ k=0,1,2,\dots, n-1, \label{E:RHOIND}\\
1 - z^2\S^2\omega_k^2(z) & \ge \frac\eta2, \ \ z\in [\bar z- t, \bar z], \ \ k=0,1,2,\dots, n-1. \label{E:ETANIND}
\end{align}
Here we choose to start the iteration with constant functions $(\rho_0(z),\omega_0(z))\equiv (\rho(\bar z),\omega(\bar z))$, $z\in [\bar z-t,\bar z]$ so that it satisfies the inductive assumptions.
From~\eqref{E:RHON}, we easily conclude
\begin{align}
\lv \rho_n(z)\rv \le \lv \rho(\bar z)\rv + \frac{CM^2}{\bar z^2 \eta} |z-\bar z|, \ \ z\in [\bar z-t, \bar z], \notag 
\end{align}
and therefore, for $t$ sufficiently small and a sufficiently large $M$ (but from now on fixed), we obtain the bound~\eqref{E:RHOIND} for $k=n$.
From~\eqref{E:OMEGAN} we easily conclude
\begin{align}\label{E:OMEGANBOUND}
\lv \omega_n(z)\rv \le \lv \omega(\bar z)\rv +  C \left(1+ 3|\omega_{n-1}|_{C^0}\right)\frac{|z-\bar z|}{\bar z}+ \frac{CM}{\bar z^3 \eta} |z-\bar z|, \ \ z\in [\bar z-t, \bar z],
\end{align}
and therefore, for $t=t(\eta,\bar z)$ sufficiently small we obtain the bound~\eqref{E:OMEGAIND} for $k=n$ using~\eqref{E:APRIORIOMEGAZERO}.
We also observe that for any $z\in[\bar z- t,\bar z]$
\begin{align}
1-\S^2z^2\omega_n^2 & = 1-\S^2z^2\omega_0^2 + \S^2z^2 \sum_{k=1}^{n} \left(\omega_{k-1}^2-\omega_k^2\right) \notag \\
& \ge 1-\S^2z^2\omega_0^2 - C\bar z \sum_{k=1}^{n}|\omega_k-\omega_{k-1}| 
\label{E:SONICBOUND}
\end{align}

Subtracting two iterates $(\omega_n,\rho_n)$ and $(\omega_{n-1},\rho_{n-1})$ 
we conclude
\begin{align}
\omega_n(z)-\omega_{n-1}(z) & =3\int_{z}^{\bar z} \frac{\omega_n-\omega_{n-1}}{\tau}\,d\tau
- 2\S^2\int_{z}^{\bar z} \left(\mathcal F(\S,\rho_{n-1},\omega_{n-1})-\mathcal F(\S,\rho_{n-2},\omega_{n-2})\right)\,d\tau,  \label{E:DIFFOMEGAN}\\
\rho_n(z)-\rho_{n-1}(z) & = 2\S^2\int_{z}^{\bar z} \left(\mathcal G(\S,\rho_{n-1},\omega_{n-1})-\mathcal G(\S,\rho_{n-2},\omega_{n-2})\right) \,d \tau, \label{E:DIFFRHON}
\end{align}
A simple algebraic manipulation and the bounds~\eqref{E:ETAZERO},~\eqref{E:APRIORIOMEGAZERO}, and~\eqref{E:APRIORIRHOZERO}  imply that there exists a constant $\tilde C=\tilde C(M,\bar z)$ such that 
for all $1\le k\le n-1$ and $z\in[\bar z- t, \bar z]$ 
\begin{align}
 \lv \mathcal F(\S,\rho_k,\omega_k) -  \mathcal F(\S,\rho_{k-1},\omega_{k-1}) \rv  & \le \frac{\tilde C}{\eta^2}
 \left(|\omega_{k}-\omega_{k-1}|+|\rho_{k}-\rho_{k-1}| \right),\label{E:ESTFITERATION}\\
 \lv \mathcal G(\S,\rho_k,\omega_k) -   \mathcal G(\S,\rho_{k-1},\omega_{k-1}) \rv  & \le  
 \frac{\tilde C}{\eta^2}\left(|\omega_{k}-\omega_{k-1}|
+|\rho_{k}-\rho_{k-1}| \right).\label{E:ESTGITERATION}
\end{align}
Allowing the constants $C$ to change from line to line, but to possibly depend on $\bar z, \eta$, 
we plug~\eqref{E:ESTFITERATION}--\eqref{E:ESTGITERATION} back into~\eqref{E:DIFFOMEGAN}--\eqref{E:DIFFRHON} and using $\bar z-t\ge \frac{\bar z}{2}$  we obtain for $k=1,2,\dots n$ 
\begin{align}
u_k(z) &\le C u_{k-1}(z) |z-\bar z|, \notag \\ 
u_k(z)& : = \lv\omega_n(z)-\omega_{n-1}(z) \rv_{C^0([z,\bar z])} + \lv \rho_n(z)-\rho_{n-1}(z)  \rv_{C^0([z,\bar z])}.\notag 
\end{align}
Choosing $t=t(\eta,\bar z)$ sufficiently small we conclude that there exists a constant $0<c<1$ such that $u_k\le c u_{k-1}$ for all $k=1,2,\dots, n$. 
By~\eqref{E:OMEGANBOUND} and~\eqref{E:SONICBOUND}
\begin{align}\label{E:SONICETAN}
1-\S^2z^2\omega_n^2 > \eta - C\bar z\sum_{k=1}^{n} c^k 
> \frac\eta2, \ \ z\ge \bar z- t,
\end{align}
for $t=t(\eta,\bar z)$ and therefore $c$ sufficiently small. Since we can choose $t$ so small that $t<\frac {\bar z}2$ bound~\eqref{E:SONICETAN} gives us~\eqref{E:ETANIND} with $k=n$. 
 By the standard arguments we pass to a limit as $n\to\infty$ and obtain the unique LP-type solution on the interval $[\bar z-t,1]$. 
\end{proof}



\begin{lemma}[No blow up before the sonic line]
Let $\S\in[2,3]$ be given and consider the unique LP-type solution $(\rho(\cdot;y_\ast),\omega(\cdot;y_\ast))$ to the left of $z=1$.
If $s(y_\ast)>0$ then
\begin{align}
\lim_{z\to s(y_\ast)}\omega(z)^2 = \frac1{\S^2s(\S)^2}. \notag 
\end{align}
\end{lemma}


\begin{proof}
Suppose the opposite. Then there exists an $\eta>0$ such that $1-z^2\S^2\omega(z)^2>\eta$ for all $z\in(s(y_\ast),1-r)$. By Lemmas~\ref{L:APRIORIODE}--\ref{L:TQUANT}
 there exists a constant $t=t(\eta,s(\S))$ such that the solution can be continued to the interval $(s(y_\ast)-t,1-r)$ and stay below the sonic line. A contradiction.
\end{proof}


\subsection{Sonic time continuity properties}\label{SS:SC}


Using a continuity argument we next show that the sonic time function $\S\to s(\S)$ is upper semi-continuous.

\begin{proposition}\label{P:LOWERSEMICONT}
Let $\S\in[2,3]$ be given and consider the unique LP-type solution to the problem~\eqref{E:RHOEQN}--\eqref{E:OMEGAEQN} to the left of $z=1$.
\begin{enumerate}
\item[{\em (a)}] (Upper semi-continuity of the sonic time). Then 
\[
\limsup_{\tilde y_\ast \to y_\ast} s(\tilde y_\ast) \le s(y_\ast),
\]
i.e. the map $\S\to s(\S)$ is upper semi-continuous. In particular, if $s(y_\ast)=0$ then the map $s(\cdot)$ is continuous at $y_\ast$. 

\item[{\em (b)}] ([Continuity of the flow away from the sonic time])
Let $\{y^n_\ast\}_{n\in\mathbb N}\subset[2,3]$ and $y_\ast\subset[2,3]$ satisfy $\lim_{n\to\infty}y^n_\ast = y_\ast$. Let 
$1-r>z>\max\{s(y_\ast),\sup_{n\in\mathbb N}s(y_\ast^n)\}$. Then 
\begin{align}
\lim_{n\to\infty}\omega(z;y_\ast^n) = \omega(z;y_\ast), \ \ \lim_{n\to\infty}\rho(z;y_\ast^n) = \rho(z;y_\ast).\notag 
\end{align}

\item[{\em (c)}] 
Let $\{y^n_\ast\}_{n\in\mathbb N}\subset[2,3]$ and $y_\ast\subset[2,3]$ satisfy $\lim_{n\to\infty}y^n_\ast = y_\ast$. Assume that there exist $0<Z<1-r$ and $\eta>0$ such that $s(y_\ast^n)<Z$ for all $n\in\mathbb N$ and the following uniform bound holds:
\begin{align}\label{E:LOWERBOUNDN}
1 - (y_\ast^n)^2 z^2 \omega(z;y_\ast^n)^2>\eta,  \ \ n\in\mathbb N, \ \ z\in[Z,1-r].
\end{align}
Then there exists a $T=T(\eta,Z)>0$ such that 
\begin{align}\label{E:CLAIMUNIFCONT}
s(y_\ast)<Z-T, \ \ s(y_\ast^n)<Z-T, \ \ n\in\mathbb N.
\end{align}

\end{enumerate}
\end{proposition}



\begin{proof}
{\em Proof of part (a).}
For any $\S\in[2,3]$, on the interval $(s(\S), 1]$ by Lemma~\ref{L:APRIORIODE} we have the a priori bounds
\begin{align}\label{E:APRIORICONT}
|\rho(z;\S)|\le \frac1{\S z}\le\frac{1}{2z}, \ \ |\omega(z,\S)| \le \frac1{\S z}\le\frac{1}{2z}, \ \ y_\ast \in [2,3], \ z\in(s(\S),1).  
\end{align}

Fix an arbitrary $\mathring z \in (s(y_\ast),1-r)$. By the definition of the sonic time $s(y_\ast)$, there exists an $\eta>0$ such that 
\begin{align}\label{E:ETA}
1 - z^2\S^2\omega(z)^2>\eta, \ \ z\in[\mathring z,1-r],
\end{align}
where $\omega(z):=\omega(z;\S)$. 
By~\eqref{E:APRIORICONT} it is clear that there exists a constant $C = C(\mathring z)$ 
such that for any $\tilde y_\ast\in[2,3]$
\begin{align}\label{E:APRIORIDELTA}
|\rho(z;\tilde y_\ast)|\le C, \ \ |\omega(z,\tilde y_\ast)| \le C, 
\ \ z\in[\mathring z, 1-r]\cap (s(\tilde y_\ast), 1-r).
\end{align}

Let $1>\delta>0$ be a small number to be specified later. Let $|\tilde y_\ast-y_\ast|<\delta$ and consider the two solutions $(\rho(z;\S),\omega(z;\S))$ and $(\rho(z;\tilde y_\ast), \omega(z;\tilde y_\ast))$ on the interval $(Z,1-r]$, 
where
\[
Z : = \max\{s(\tilde y_\ast), \mathring z\}. 
\]
Clearly both solutions are well-defined on $( Z,1-r]$..

Our goal is to show that $Z=\mathring z$ if $\delta$ is sufficiently small. To that end, assume the opposite, i.e. $Z=s(\tilde y_\ast)$.
In the rest of the proof the constant $C$ may change from line to line, but may depend only on $\mathring z$ and $y_\ast$.

For any $z\in(Z,1-r)$ integrating~\eqref{E:RHOEQN}--\eqref{E:OMEGAEQN} over $[z,1-r]$ to obtain
\begin{align}
\omega(z) & = \omega(1-r) + \int_{z}^{1-r} \frac{3\omega-1}{\tau}\,d\tau 
 - 2\S^2\int_{z}^{1-r} \frac{\tau \omega^2(\rho-\omega)}{1-\S^2\tau^2\omega^2}\, d\tau,  \label{E:OMEGAEQNINTEGRATED}\\
 \rho(z) & =\rho(1-r)  + 2\S^2\int_{z}^{1-r} \frac{\tau \omega\rho(\rho-\omega)}{1-\S^2\tau^2\omega^2}\, d\tau.
 \label{E:RHOEQNINTEGRATED}
\end{align}
For any $y_\ast,\tilde y_\ast\in [2,3]$ denote the corresponding LP- type solutions by $(\rho,\omega)$ and $(\tilde\rho,\tilde\omega)$ respectively. From~\eqref{E:OMEGAEQNINTEGRATED}--\eqref{E:RHOEQNINTEGRATED} we obtain
\begin{align}
\omega(z)-\tilde\omega(z) & = \omega(1-r)-\tilde\omega(1-r)+3\int_{z}^{1-r} \frac{\omega-\tilde\omega}{\tau}\,d\tau
- 2\S^2\int_{z}^{1-r} \left(\mathcal F(\S,\rho,\omega)-\mathcal F(\tilde y_\ast,\tilde\rho,\tilde\omega)\right)\,d\tau,  \label{E:DIFFOMEGA}\\
\rho(z)-\tilde\rho(z) & = \rho(1-r) - \tilde\rho(1-r)+2\S^2\int_{z}^{1-r} \left(\mathcal G(\S,\rho,\omega)-\mathcal G({\tilde y}_\ast,\tilde\rho,\tilde\omega)\right) \,d \tau, \label{E:DIFFRHO}
\end{align}
where the nonlinearities $\mathcal F$ and $\mathcal G$ are defined in~\eqref{E:FDEF} and~\eqref{E:GDEF}.

We let
\begin{align}
g(z): = \lv \omega(z)-\tom(z)\rv + \lv \rho(z)-\tr(z)\rv.\notag 
\end{align}
Since 
$1 - \tilde y_\ast^2 z^2 \tom^2 = 1 - \S^2 z^2 \omega^2 + z^2\left(\omega^2-\tom^2\right)\tilde y_\ast^2 +z^2\omega^2\left(\S^2-\tilde y_\ast^2\right)$, from~\eqref{E:APRIORIDELTA} we conclude
\begin{align}
1 - \tilde y_\ast^2 z^2 \tom^2  & \ge 1 - \S^2 z^2 \omega^2 - C \left(g(z)+\lv \S - \tilde y_\ast\rv\right) \notag \\
& \ge \eta - C \left(g(z)+\lv \S - \tilde y_\ast\rv\right). \label{E:TILDEETA}
\end{align}
We let 
\be
\bar\eta(z) : = \eta - C \left(g(z)+\lv \S - \tilde y_\ast\rv\right).
\label{E:BARETADEF}
\ee
Clearly, for $\delta>0$  and $|1-r-z|$ sufficiently small, we have $\bar\eta>\frac\eta2$ by continuity. Let  
\begin{align}\label{E:BARZDEF}
\bar Z : = \inf_{Z<z<1-r} \left\{\bar\eta(z)>\frac\eta2\right\}.
\end{align}
For any $z\ge \bar Z$ a simple algebraic manipulation and the bounds~\eqref{E:ETA},~\eqref{E:TILDEETA}, and~\eqref{E:APRIORIDELTA} 
give
\begin{align}
 \lv \mathcal F(\S,\rho,\omega) -  \mathcal F({\tilde y}_\ast,\tr,\tom) \rv  & \le \frac{C}{\eta\bar\eta}\left(|\omega-\tom|
+|\rho-\tr| + |\S-{\tilde y}_\ast|\right), \ \ |y_\ast-\tilde y_\ast|<\delta, \label{E:ESTF}\\
 \lv \mathcal G(\S,\rho,\omega) -  \mathcal G({\tilde y}_\ast,\tr,\tom) \rv  & \le \frac{C}{\eta\bar\eta}\left(|\omega-\tom|
+|\rho-\tr| + |\S-{\tilde y}_\ast|\right), \ \ |y_\ast-\tilde y_\ast|<\delta.\label{E:ESTG}
\end{align}
The  identities~\eqref{E:DIFFOMEGA}--\eqref{E:DIFFRHO} and estimates~\eqref{E:ESTF}--\eqref{E:ESTG} now give 
\begin{align}
g(z) & \le g(1-r) + \frac{C}{\eta\bar\eta} |\S-\tilde y_\ast| +  \frac3{\mathring z} \int_z^{1-r} |\omega(\tau)-\tom(\tau)|\,d\tau + \frac{C}{\eta\bar\eta} \int_z^{1-r} g(\tau)\,d\tau  \notag \\
& \le  g(1-r) + \frac{C}{\eta\bar\eta} |\S-\tilde y_\ast|+ \frac{C}{\eta\bar\eta} \int_z^{1-r} g(\tau)\,d\tau, \ \ z\in[\bar Z, 1-r] .\label{E:GBOUND}
\end{align}
It follows by a Gr\"onwall argument and~\eqref{E:BARZDEF} that
\begin{align}
g(z) & \le \left(g(1-r) + \frac{C}{\eta\bar\eta} |\S-\tilde y_\ast|\right)  \frac{C}{\eta\bar\eta} e^{ \frac{C}{\eta\bar\eta}(1-r-z)}, \notag\\
& \le \left(g(1-r) + \frac{C}{\eta^2} |\S-\tilde y_\ast|\right)  \frac{C}{\eta^2} e^{ \frac{C}{\eta^2}(1-r-z)}, \ \ z\in[\bar Z, 1-r].\notag 
\end{align}
We note that for any given $\delta'>0$, there exists a $\delta>0$ such that $g(1-r)<\delta'$ for all $|\S-\tilde y_\ast|<\delta$.
Therefore, for any given $\epsilon>0$ we can choose a $\delta=\delta(\eta,\epsilon)$ 
sufficiently small so that 
for all $|\S-\tilde y_\ast|<\delta$ we have the bound 
\be
g(z)<\epsilon, \ \ \bar Z < z \le 1-r. \notag 
\ee
However, by~\eqref{E:BARETADEF} we then have
\begin{align}
\bar\eta(\bar Z) \ge \eta - C\left(\epsilon+\delta\right)>\frac\eta2, \ \ \text{ for $\delta$ sufficiently small}. \notag 
\end{align}
This is only possible if $\bar Z = Z$.
This gives a uniform lower bound for $1 - \tilde y_\ast^2 z^2 \tom^2$ on $(Z,1-r]$ and this contradicts the assumption $Z=s(\tilde y_\ast)$. Therefore $Z=\mathring z$ and $s(\tilde y_\ast)$ is strictly smaller that $\mathring z$ by Lemma~\ref{L:TQUANT}. 
Since $\mathring z>s(\S)$ is arbitrary it follows that
for any $\varepsilon>0$ there exists a $\delta>0$ such that $|\tilde y_\ast-\S|<\delta$ implies $s(\tilde y_\ast)-s(\S)<\varepsilon$, which is equivalent to upper semi-continuity. If $s(\S)=0$ this implies the continuity of $s(\cdot)$ at $\S$. 

\noindent
{\em Proof of part (b).}
Since $z$ is a fixed distance away from the sonic time $s(y_\ast)$, there exists a constant $\tilde\eta>0$ such that $1-\tau^2y_\ast^2\omega(\tau;y_\ast)^2>\tilde\eta$ for all $\tau\in[z,1-r]$. 
By the proof of part (a) there exists a neighbourhood of $y_\ast$ depending on $\tilde\eta$ and $z$, such that all LP- type solutions launched from that neighbourhood have a sonic time strictly less than $s(y_\ast)$. The claim now follows from~\eqref{E:GBOUND}.

\noindent
{\em Proof of part (c).} This is again a consequence of the arguments in the proof of part (a). By Lemma~\ref{L:TQUANT} it is clear that there exists a $T=T(\eta,Z)$ such that  $s(y_\ast^n)<Z-T$ for all $n\in\mathbb N$. On the other hand, due to the lower bound~\eqref{E:LOWERBOUNDN} and the proof of part (a) there exists a $\delta=\delta(\eta,Z)$ such that for all $|\tilde y_\ast-y_\ast^n|<\delta$ the sonic time $s(\tilde y_\ast)<Z-T$ for some, possibly smaller time $T=T(\eta,Z)>0$. Letting $n$ large enough, this concludes the proof.
\end{proof}

\subsection{The set $Y$ and the minimality property}\label{SS:SETY}

We now partition the interval $[2,3]$  in the the sets $\X,\Y,\Z$ that will play an important role in our analysis.
We let
\begin{align}\label{E:XDEF}
\X : = \left\{y_\ast\in[2,3]\,\Big| \inf_{z\in(s(y_\ast),1)}\omega(z;y_\ast)>\frac13 \right\},
\end{align}
\begin{align}\label{E:YDEF}
\Y : = \left\{y_\ast\in[2,3]\,\big| \ \exists z\in (s(y_\ast), 1) \ \text{such that } \omega(z;y_\ast)=\frac13\right\}
\end{align}
\begin{align}\label{E:ZDEF}
\Z: = \left\{y_\ast\in[2,3]\,\Big| \omega(z;y_\ast)>\frac13 \ \text{ for all } \ z\in(s(y_\ast),1) \ \text{ and } \ \inf_{z\in(s(y_\ast),1)}\omega(z;y_\ast)\le\frac13\right\}.
\end{align}
Clearly
$
[2,3] = \X \cup \Y \cup \Z
$
and the sets $\X,\Y,\Z$ are disjoint. We introduce the fundamental set $Y\subset \Y$
\begin{align}\label{E:YDEFA}
Y : = \left\{y_\ast\in[2,3]\,\big| \ \exists z\in (s(\tilde y_\ast), 1) \ \text{such that } \omega(z;\tilde y_\ast)=\frac13 \ \text{ for all } \ \tilde y_\ast\in[y_\ast,3]\right\},
\end{align}
and let
\begin{align}
\bar y_\ast &: = \inf_{y_\ast\in Y} y_\ast.\label{E:BARYDEF} 
\end{align}


The next statement shows that sets $\Y$ and $\X$ are not empty.
\begin{lemma}\label{L:TWOANDTHREE}
\begin{enumerate}
\item[{\em (a)}] There exists an $\epsilon>0$ such that  $(3-\epsilon,3]\subset Y\subset \Y$
\item[{\em (b)}] $2\in \X$.
\end{enumerate}
\end{lemma}


\begin{proof}
{\em Proof of part (a).} By the mean value theorem, we write $\omega(z;y_\ast)$ as 
\[
\omega(z;y_\ast) = \frac{1}{y_\ast} + \omega'(\bar z; y_\ast) (z-1), \ z\in (s(y_\ast), 1) 
\] 
for some $\bar z\in (z,1)$.  From \eqref{E:LPCONDITION}, we have $ \omega'(1; y_\ast) =1-\frac{2}{y_\ast} $ with  $ \omega'(1; 3) =\frac13 $. By Theorem \ref{T:ANALYTICITY} and Lemma \ref{L:ANALYTIC}, there exist small enough $r>0$ and $\epsilon_1>0$ such that $\omega'(z; y_\ast)>\frac16$ for all $z\in(1-r,1]$ and $y_\ast\in (3-\epsilon_1,3]$. Then for $z\in(1-r,1]$ and $y_\ast\in (3-\epsilon_1,3]$, we have 
\[
\omega(z;y_\ast) \le \frac{1}{y_\ast} +\frac16 (z-1). 
\]
Note $ \frac{1}{y_\ast} + \frac16 (z-1) = \frac 13$ when $z=z^\ast(y_\ast)= 1- \frac{2(3-y_\ast)}{y_\ast}$. Therefore for all $y_\ast \in (3-\epsilon, 3]$ with 
$\epsilon = \min \{ \epsilon_1, \frac{3r}{2+r}\}$, there exists $\tilde z\ge z^\ast(y_*)$ such that $\omega(\tilde z;y_\ast)=\frac13$, which shows $(3-\epsilon,3]\subset Y\subset \Y$.

\noindent{\em Proof of part (b).} Let $y_\ast=2$ and denote $\omega(\cdot;2)$ by $\omega$. First we rewrite  \eqref{E:OMEGAEQN} as 
\begin{align*}
z\omega' & = 1-2\omega -\omega + \frac{2\S^2z^2\omega^2}{1-\S^2z^2\omega^2}(\rho-\omega) \\
&= 1- 2\omega - \omega \left[ \frac{1- (\S z\rho)^2 + (\S z \rho -\S z \omega)^2 }{1-\S^2z^2\omega^2}\right]. 
\end{align*}
By Lemma \ref{L:APRIORIODE}, we have $z\omega' \leq 1-2\omega$, which implies that $\omega>\frac12$ is an invariant set. On the other hand from \eqref{E:LPCONDITION} we know that 
\[
\omega(1)=\frac12, \ \ \  \omega'(1)=0,  \ \ \ \omega''(1)  = \frac{1}{2} 
\]
and hence $\omega >\frac12$ on $(1-\eta, 1)$ for sufficiently small $\eta>0$. Therefore, we conclude that $\inf_{z\in(s(2),1)}\omega(z;2)=\frac12$ and $2\in \X$.
\end{proof}


\begin{definition}
For any $y_\ast>0$ we define
\begin{align}
z_{\frac13} &  = z_{\frac13}(y_\ast) : = \inf \left\{z \in (s(y_\ast), 1) \, \big| \ 
\ \omega(\tau;y_\ast)>\frac13 \ \text{ for } \ \tau \in(z,1) \right\}. \label{E:YTHIRDDEFALTZ}
\end{align}
\end{definition}

\begin{remark}\label{R:SETY}
Geometrically, if we follow the solution curve $z\mapsto \omega(z;y_\ast)$ starting at $z=1$ and going to the left, point $z_{\frac13}$ is the first time this curve crosses the value $\frac13$ from above.
By the definition of $\Y$, for any $y_\ast\in \Y$ there exists an $z_{\frac13}\in(s(y_\ast), 1]$ such that 
\begin{align}
\omega(z_{\frac13};y_\ast) = \frac13. \notag 
\end{align}

Therefore, for any  $y_\ast\in[2,3]$ we have the following possibilities:
\begin{enumerate}
\item 
$y_\ast\in \Y$ and therefore $z_{\frac13}(y_\ast)>s(y_\ast)\ge0$.
\item
$y_\ast\in [2,3]\setminus \Y$ and $z_{\frac13}(y_\ast)=s(y_\ast)>0$.  In this case we must have 
\be\label{E:WLOWERBOUNDZ}
\omega(z_{\frac13}(y_\ast);y_\ast)=\omega(s(y_\ast);y_\ast)>\frac13;
\ee
otherwise if $\omega(z_{\frac13}(y_\ast);y_\ast)=\frac13$ then $\omega(z_{\frac13}(y_\ast);y_\ast) z_{\frac13}(y_\ast)=\frac13 z_{\frac13}(y_\ast)<1$ and thus $s(y_\ast)<z_{\frac13}(y_\ast)$.
\item 
$y_\ast\in [2,3]\setminus \Y$ and $z_{\frac13}(y_\ast)=s(y_\ast)=0$.
\end{enumerate}
\end{remark}


The sets $\Y$ and $\Z$ enjoy several important properties which we prove in the next lemma.
\begin{lemma}\label{L:PREPZ}
\begin{enumerate}
\item[{\em (a)}] For any $y_\ast\in \Y\cup \Z$ we have 
\begin{align}
\omega(z;y_\ast)&<\rho(z;y_\ast), \ \ z\in (s(y_\ast), 1), \label{E:WLESSTHANRZ}\\
\omega(z;y_\ast) & <\frac13, \ \ z\in (s(y_\ast), z_{\frac13}(y_\ast)),
\label{E:WLESSTHANTHIRDZ}
\end{align}
where~\eqref{E:WLESSTHANTHIRDZ} is considered trivially true in the case $s(y_\ast)=z_{\frac13}(y_\ast)$.
\item[{\em (b)}]
For any $y_\ast\in \Y$ we have $\omega'(z_{\frac13}(y_\ast);y_\ast)>0$. Moreover, the set $\Y$ is relatively open in $[2,3]$. 
\end{enumerate}
\end{lemma}

\begin{proof}
{\em Proof of~\eqref{E:WLESSTHANRZ}.}
Let $y_\ast\in \Y\cup \Z$.
By~\eqref{E:LPCONDITION} we know that $\omega(z)<\rho(z)$ for all $z\in[1-\bar r,1)$ for some $\bar r\le r$, where $r$ is given by Theorem~\ref{T:ANALYTICITY}. By way of contradiction, assume that there exists $z_c\in (s(y_\ast), 1)$ such that 
\begin{align}\label{E:YCDEFZ}
\omega(z_c) = \rho(z_c), \ \ \rho(z)>\omega(z), \ z\in(z_c,1).
\end{align}
We distinguish three cases.

\noindent
{\em Case 1: $z_c\in(z_{\frac13},1)$.}
In this case we conclude from~\eqref{E:OMEGAEQN} that 
$\omega'(z_c)<0$ by~\eqref{E:YCDEFZ} and~\eqref{E:YTHIRDDEFALTZ} and from~\eqref{E:RHOEQN} and~\eqref{E:YCDEFZ} $\rho'(z_c)=0$. In particular
$(\rho-\omega)'(z_c)>0$ and locally strictly to the left of $z_c$ we have 
\begin{align}\label{E:CONTR1Z}
\omega'<0, \ \rho-\omega <0, \ \rho'>0, \ \omega>\frac13.
\end{align}
We note that $\rho'>0$ follows from $\rho-\omega<0$ and~\eqref{E:RHOEQN}, while $\omega>\frac13$ is implied by the assumption $z_c\in(z_{\frac13},1)$. It is easy to see that the 
conditions~\eqref{E:CONTR1Z} are dynamically trapped, and since $\omega'<0$ we conclude that $\omega$ stays strictly bounded away from $\frac13$ from above for all 
$z\in (z_{\frac13},1)$. This is a contradiction to the assumption $y_\ast\in \Y\cup \Z$.

\noindent
{\em Case 2: $z_c=z_{\frac13}$.}
In this case $y_\ast\in \Y$ necessarily and 
\begin{align}\label{E:CONTR2Z}
\omega(z_{\frac13}) = \rho(z_{\frac13}) = \frac13.
\end{align}
However, since $\rho-\omega>0$ for $z\in (z_{\frac13},1)$ equation~\eqref{E:RHOEQN} implies $\rho'<0$ on $(z_{\frac13},1)$ and therefore $\rho(z_{\frac13})>\rho(1)=\frac1{y_\ast} \ge\frac13$, since $y_\ast [2,3]$. This is a contradiction to~\eqref{E:CONTR2Z}.

\noindent
{\em Case 3: $z_c \in(s(y_\ast),z_{\frac13})$.} In this case $y_\ast\in \Y$ necessarily.
Since $z_c<z_{\frac13}$ we know that $\rho-\omega>0$ locally around $z_{\frac13}$. Therefore, by~\eqref{E:OMEGAEQN}--\eqref{E:RHOEQN} and~\eqref{E:YTHIRDDEFALTZ} we have 
\begin{align}\label{E:CONTR3Z}
\omega'>0, \ \rho-\omega>0, \ \omega<\frac13 \ \ \text{ on } \ (z_{\frac13}-\varepsilon, z_{\frac13})
\end{align}
for a sufficiently small $\varepsilon>0$. The region described by~\eqref{E:CONTR3Z} is dynamically trapped  and we conclude that $\rho-\omega>0$ on $(s(y_\ast),z_{\frac13})$. This is a contradiction, thus completing the proof of~\eqref{E:WLESSTHANRZ}. Inequality~\eqref{E:WLESSTHANTHIRDZ} follows by a similar argument, since the property~\eqref{E:CONTR3Z} is dynamically preserved and all three properties are easily checked to hold locally to the left of $z_{\frac13}(y_\ast)$.

\noindent
{\em Proof of part (b).}
For any $y_\ast\in \Y$ by part (a)
and~\eqref{E:OMEGAEQN} we have $\omega'(z_{\frac13}(y_\ast);y_\ast)>0$. Therefore there exists a $\delta>0$ sufficiently small so that $\omega(z;y_\ast)<\frac13$ for all $z\in(z_{\frac13}(y_\ast)-\delta,z_{\frac13}(y_\ast))$. By the proof of Proposition~\ref{P:LOWERSEMICONT} there exists a small neighbourhood of $y_\ast$ such that $\omega(z;y_\ast)<\frac13$ for some $z\in (z_{\frac13}(y_\ast)-\delta,z_{\frac13}(y_\ast))$. Therefore $\Y$ is open.
\end{proof}


Another remarkable feature of the sets $\Y$ and $\Z$ is the following uniform lower bound on the distance to the sonic line at points $z$ larger than $z_{\frac13}(\S)$.

\begin{lemma}\label{L:LOWERBOUND}
There exists a constant $\eta>0$ such that 
\begin{align}
1 - y_\ast^2 z^2 \omega(z;y_\ast)^2 > \eta, \ \ y_\ast \in \Y\cup \Z, \ \ z\in (z_{\frac13}(y_\ast),1-r], \notag 
\end{align}
where $r$ is the constant given in Theorem~\ref{T:ANALYTICITY}.
\end{lemma}


\begin{proof}
It is clear that there exists an $\eta>0$ such that $1 - y_\ast^2 z^2 \omega(z;y_\ast)^2 > \eta$ at $z=1-r$ for all $y_\ast\in[2,3]$. By~\eqref{E:WLESSTHANRZ} and~\eqref{E:ZRHOMONOTONE}  and since $\omega(z;\bar y_\ast)>\frac13$ for all $z\in (z_{\frac13}(y_\ast),1-r]$,  we have
\begin{align}
1 - y_\ast^2 z^2\omega(z;y_\ast)^2 > 1 - y_\ast^2 z^2\rho(z;y_\ast)^2> 1-  y_\ast^2 (1-r)^2\rho(1(1-r);y_\ast)^2>\eta, \  \ z\in (z_{\frac13}(y_\ast),1-r], \notag 
\end{align}
for all $y_\ast\in \Y\cup \Z$.
\end{proof}

\begin{remark}
It is easily seen from the proof that the uniform-in-$y_\ast$ lower bound from Lemma~\ref{L:LOWERBOUND} holds as long as $\omega(z;y_\ast)\ge0$.
\end{remark}



We now use the previously shown regularity properties to obtain the key result, which states that $s(\bar y_\ast)=0$, i.e.
the LP-type solution associated with $\bar y_\ast=\inf Y$ extends to the left from $z=1$ all the way to $z=0$.

\begin{proposition}[Existence up to the origin]\label{P:EXISTENCEY}
Recall $\bar y_\ast$ defined in~\eqref{E:BARYDEF}. The solution $(\omega(z;\bar y_\ast), \rho(z;\bar y_\ast))$ exists on $(0,1]$, i.e. 
\begin{align}
s(\bar y_\ast)=0. \notag 
\end{align}
\end{proposition}

\begin{proof}

\noindent
{\em Case 1: $z_{\frac13}(\bar y_\ast)=0$.} 
In this case we are done as by definition $0\le s(\bar y_\ast)\le z_{\frac13}(\bar y_\ast)=0$.

\smallskip
\noindent
{\em Case 2: $z_{\frac13}(\bar y_\ast)>s(\bar y_\ast)>0$.}  In this case $\bar y_\ast\in Y$ and
\[
\omega(z_{\frac13}(\bar y_\ast);\bar y_\ast) = \frac13
\]
There exists a $\delta>0$ such that 
\be\label{E:WSMALLERTHIRDZ}
\omega(z;\bar y_\ast)<\frac13, \ \  z\in (z_{\frac13}(\bar y_\ast)-\delta,z_{\frac13}(\bar y_\ast)),
\ee 
by Lemma~\ref{L:PREPZ}.
Let now $\{y^n_\ast\}_{n\in\mathbb N}\subset [2,3]\setminus Y$ satisfy 
\begin{align}
\lim_{n\to\infty} y_\ast^n = \bar y_\ast. \notag 
\end{align}
Note that we choose $\{y^n_\ast\}_{n\in\mathbb N}\subset \X\cup \Z$, which is possible by the openess of $Y$, $\Y$, and the definition of $\bar y_\ast$, see~\eqref{E:BARYDEF}.
Since by Proposition~\ref{P:LOWERSEMICONT} $s(\bar y_\ast)\ge\limsup_{n\to\infty}s(y_\ast^n)$ and $z_{\frac13}(\bar y_\ast)>s(\bar y_\ast)$ it follows that there exists $N\in\mathbb N$ sufficiently large such 
that (after passing to a subsequence) $s(y_\ast^n)<z_{\frac13}(\bar y_\ast)-\delta$ for all $n>N$, where we have chosen a possibly smaller $\delta$. By part (b) of Proposition~\ref{P:LOWERSEMICONT} it follows that for any $z\in(z_{\frac13}(\bar y_\ast)-\delta,z_{\frac13}(\bar y_\ast))$
$\omega(z;\bar y_\ast)=\lim_{n\to\infty}\omega(z;y_\ast^n) \ge\frac13$, a contradiction to~\eqref{E:WSMALLERTHIRDZ}.

\smallskip
\noindent
{\em Case 3: $z_{\frac13}(\bar y_\ast)=s(\bar y_\ast)>0$.} In this case $\bar y_\ast\in [2,3]\setminus Y$ and therefore, by definition of $\bar y_\ast$ we have $\bar y_\ast\in\X\cup\Z$. 
By~\eqref{E:WLOWERBOUNDZ} we have 
$\omega(s(\bar y_\ast);\bar y_\ast)>\frac13$. Let now $\{y^n_\ast\}_{n\in\mathbb N}\subset Y$ satisfy 
\begin{align}
\lim_{n\to\infty} y_\ast^n = \bar y_\ast. \notag 
\end{align} 
Define
\begin{align}
\bar z_{\frac13} : = \limsup_{n\to\infty} z_{\frac13}(y_\ast^n) \notag
\end{align}

We need to distinguish two subcases.

\noindent
{\em Subcase 1: $\bar z_{\frac13}>0$.}
By Lemma~\ref{L:PREPZ} and Remark~\ref{R:SETY} we have $s(y_\ast^n)<z_{\frac13}(y_\ast^n)$. By Lemma~\ref{L:LOWERBOUND} there exists a positive number $\eta$ such that 
\begin{align}
1 - z^2(y_\ast^n)^2\omega(z;y_\ast^n)^2 > \eta, \ \ n\in\mathbb N, \ \  z\in [z_{\frac13}(y_\ast^n), 1-r]. \notag 
\end{align}
Upon passing to a subsequence $\{y_\ast^n\}_{n\in\mathbb N}$ such that $\lim_{n\to\infty}z_{\frac13}(y_\ast^n)=\bar z_{\frac13}$, by part (c) of Proposition~\ref{P:LOWERSEMICONT}  we conclude that there exists a $T=T(\eta,\bar z_{\frac13})>0$ such that $s(\bar y_\ast), s(\bar y_\ast^n)<\bar z_{\frac13}-T$, $n\in\mathbb N$.
In particular, for any $z\in(\bar z_{\frac13}-T,\bar z_{\frac13})$ we conclude by part (b) of Proposition~\ref{P:LOWERSEMICONT} that $\omega(z;\bar y_\ast)=\lim_{n\to\infty}\omega(z;y_\ast^n)\le\frac13$, a contradiction to $\bar y_\ast \notin Y$.

\noindent
{\em Subcase 2: $\bar z_{\frac13}=0$.}
For any fixed $Z>0$ we can apply the argument from Subcase 1 to conclude that the $s(\bar y_\ast)<Z$. Therefore $s(\bar y_\ast)=0$ in this case.
\end{proof}


\begin{lemma}[Continuity of $\Y\ni y_\ast\mapsto z_{\frac13}(y_\ast)$]\label{L:ZONETHIRDCONTINUITY}
The map 
\[
\Y\ni y_\ast \mapsto z_{\frac13}(y_\ast) 
\]
is continuous and 
\be\label{E:ZONETHIRDCONT}
\lim_{Y\ni y\to \bar y_\ast} z_{\frac13}(y) = 0 = z_{\frac13}(\bar y_\ast).
\ee
\end{lemma}


\begin{proof}
Let $y_\ast\in \Y$.  By Lemma~\ref{L:LOWERBOUND} there exists a $\delta>0$ such that $s(\tilde y_\ast)<z_{\frac13}(y_\ast)-\delta$ for all $\tilde y_\ast$ in an open neighbourhood of $y_\ast$. Since by part (b) of Lemma~\ref{L:PREPZ} 
$\omega'(z_{\frac13}(y_\ast);y_\ast)>0$, we may now use the Implicit Function Theorem to conclude that the map 
$y_\ast\mapsto z_{\frac13}(y_\ast)$ is in fact $C^1$.

To show~\eqref{E:ZONETHIRDCONT} assume the opposite: there exists a sequence $\{y_\ast^n\}_{n\in\mathbb N}\subset Y\subset\Y$  such that $\lim_{n\to \infty} y_\ast^n=\bar y_\ast$, but 
\begin{align}
\alpha:=\liminf_{n\to\infty} z_{\frac13}(y_\ast^n)>0.\notag
\end{align}
Upon passing to a subsequence, we may assume without loss of generality that
$\lim_{n\to\infty} z_{\frac13}(y_\ast^n)=\alpha$ and 
 $z_{\frac13}(y_\ast^n)>\frac\alpha2$ for all $n\in\mathbb N$. By Lemma~\ref{L:LOWERBOUND} and Proposition~\ref{P:LOWERSEMICONT}  there exists an $\epsilon=\epsilon(\alpha,\eta)$ (here $\eta$ is the constant from Lemma~\ref{L:LOWERBOUND})  such that $s(y^n_\ast)<z_{\frac13}(y_\ast^n)-3\epsilon$ for all $n\in\mathbb N$.
We infer that (upon possibly passing to a subsequence) $\omega(z;y_\ast^n)<\frac13$ for all $z\in[\alpha-2\epsilon,\alpha-\epsilon]$. Thus by continuity of the map $[2,3]\ni y_\ast\mapsto \omega(z;y_\ast)$ we conclude that 
$\omega(z;\bar y_\ast)\le\frac13$ for $z\in[\alpha-2\epsilon,\alpha-\epsilon]$,  which implies $\bar y_\ast
\in Y$. By part (b) of Lemma~\ref{L:PREPZ} there is an open neighbourhood of $\bar y_\ast$ that belongs to $Y$, which is a contradiction to the minimality property of $\bar y_\ast$ and part (b) of Lemma~\ref{L:TWOANDTHREE}.
\end{proof}



\subsection{Properties of the solution from the origin to the right}\label{SS:LEFT}

In order to complete the intersection argument in Section~\ref{S:INTERSECTION} we must better understand the solutions emanating from $z=0$ to the right.
Recall that $(\rho_-(\cdot;\rho_0),\omega_-(\cdot;\rho_0))$ is the unique local solution to~\eqref{E:RHOEQN}--\eqref{E:OMEGAEQN} satisfying the boundary conditions 
\be\label{E:LEFT}
\rho_-(0) = \rho_0>0, \ \ \omega_-(0)= \frac13;
\ee
existence and uniqueness are given by Theorem~\ref{T:ANALYTICITYLEFT}.
Let $s_-(\rho_0)$ denote the sonic time (from the left), i.e.
\begin{align}
s_-(\rho_0) : = \sup_{z\ge0} \left\{z \, \big| \ y_\ast z \omega_-(z;\rho_0)<1\right\}. \notag
\end{align}
We then have the following a priori bounds on $(\rho_-,\omega_-)$.

\begin{lemma}\label{L:APRIORI}
Let $\rho_-(0)=\rho_0>\frac13$, $\omega_-(0)=\frac13$, and $y_\ast\in[2,3]$. The solution $(\rho_-(z;\rho_0),\omega_-(z;\rho_0))$ to~\eqref{E:RHOEQN}--\eqref{E:OMEGAEQN} with the initial data~\eqref{E:LEFT} exists on
the interval $[0,s_-(\rho_0))$ and satisfies the following bounds:
\begin{align}
\omega_-(z;\rho_0)&>\frac13, \ \ z\in [0,s_-(\rho_0)) \label{E:APRIORI1}\\
\rho_-(z;\rho_0) + \omega_-(z;\rho_0) & < \rho_0 + \frac13,  \ \ z\in [0,s_-(\rho_0))\label{E:APRIORI2}\\
\rho_-(z;\rho_0) \omega_-(z;\rho_0)  &< \frac13 \rho_0, \ \ z\in [0,T_\ast)\label{E:APRIORI3}\\
\rho_-(z;\rho_0)& >\omega_-(z;\rho_0), \ \  z\in [0,s_-(\rho_0))\label{E:APRIORI4}\\
\rho_-'(z;\rho_0)&<0, \ \ z\in [0,s_-(\rho_0)).\label{E:APRIORI5}
\end{align}
\end{lemma}

\begin{proof}
We suppress the $\rho_0$-dependence in the notation for $(\rho_-,\omega_-)$.

\noindent
{\em Proof of~\eqref{E:APRIORI1}.}
Since $\omega''(0)>0$ for $\rho_0>\frac13$ for $0< z\ll1$ (by~\eqref{E:BCATZEROOMEGA}) it is clear that~\eqref{E:APRIORI1} is true for any sufficiently small $z>0$. 
Suppose now that~\eqref{E:APRIORI1} is wrong and let $0<z_1<s_-(\rho_0)$ be the first $z_1$ such that $\omega_-(z_1)=\frac13$ and $\frac13<\omega_-(z)$ for $0<z<z_1$. Then $\omega_-'(z_1)\leq 0$. First suppose $\omega_-'(z_1)<0$. Then from \eqref{E:RHOEQN}--\eqref{E:OMEGAEQN} we deduce that $\rho_-(z_1)<\omega_-(z_1)=\frac13$ and $\rho'_-(z_1)>0$. Hence, there should exist $0<z_2<z_1$ such that $\rho_-'(z_2)=0$ and $\rho_-(z_2)< \rho_-(z_1)<\frac13$. Then from \eqref{E:RHOEQN}, $\omega_-(z_2)=\rho_-(z_2)$, which is a contradiction to the definition of $z_1$. Next let $\omega_-'(z_1)=0$. Then $\rho_-(z_1)=\omega_-(z_1)=\frac13$ and also $\rho'(z_1)=0$. Since $z_1$ is away from the sonic line, $(\rho_-,\omega_-,)$ is smooth and the conditions $\rho_-(z_1)=\omega_-(y_1)=\frac13$ and $\omega_-'(z_1)=\rho_-'(z_1)=0$ give $\rho_-=\omega_-=\frac13$ in an open neighborhood,  which is a contradiction. Here we have used uniqueness and the existence of the Friedman solution $(\rho_F,\omega_F)\equiv(\frac13,\frac13)$, see Remark~\ref{R:FRIEDMAN}.

\noindent
{\em Proof of~\eqref{E:APRIORI2}.}
This follows from 
\be
(\rho_-+\omega_-)' =  \frac{1-3\omega_-}{z}  -  \frac{2\S^2z\omega_-}{ 1- (\S z\omega_-)^2} (\rho_--\omega_-)^2  \notag
\ee
which is negative  for $0<z<s_-(\rho_0)$ since $\omega_->\frac13$. 


\noindent
{\em Proof of~\eqref{E:APRIORI3}.}
This follows from 
\be\label{RW}
(\rho_-\omega_-)' = \frac{\rho_-(1-3\omega_-)}{z}
\ee
which is negative  for $0<z<s_-(\rho_0)$  since $\omega_->\frac13$.

\noindent
{\em Proof of~\eqref{E:APRIORI4} and~\eqref{E:APRIORI5}.}
This follows from 
\be
(\rho_--\omega_-)' =  \frac{3\omega_--1}{z}  - \frac{2\S^2z\omega_-(\rho_-+\omega_-)}{ 1- (\S z\omega_-)^2} (\rho_--\omega_-) \notag
\ee
by integrating in $z$. Claim~\eqref{E:APRIORI5} follows from \eqref{E:RHOEQN}.
\end{proof}


\begin{lemma}
Let $\rho_0>\frac13$ be given and consider the unique solution $(\rho_-(z;\rho_0),\omega_-(z;\rho_0))$ to the initial-value problem~\eqref{E:RHOEQN}--\eqref{E:OMEGAEQN},~\eqref{E:LEFT}.
Assume that  $\rho_-(z_0;\rho_0)>\frac1{y_\ast z_0}$ for some $z_0\in(0,s_-(\rho_0))$. Then 
\begin{align}
\rho_-(z;\rho_0) >\frac{1}{y_\ast z}, \ \ z\in[z_0,s_-(\rho_0)).\notag
\end{align}
\end{lemma}

\begin{proof}
Just like in the proof of Lemma~\ref{L:APRIORIODE} we consider
\begin{align}
f_-(z) : = 1 - \S z \rho_-(z).\notag
\end{align} 
Equation~\eqref{E:FEQN} then reads
\begin{align}\label{E:FEQNMINUS}
f_-'(z)  + f_-(z) \frac{2z\S^2\omega_-\rho_-}{1-\S^2z^2\omega_-^2} = - \frac{\S (1-\S z \omega_-)^2 \rho_-}{1-\S^2z^2\omega_-^2}
\end{align}
By our assumptions $f_-(z_0)<0$. Since the right-hand side of~\eqref{E:FEQNMINUS} is negative, we conclude
\[
\frac{d}{dz}\left(f_-(z) \exp\left(\int_{z_0}^z \frac{2\tau\S^2\omega_-\rho_-}{1-\S^2\tau^2\omega_-^2}\,d\tau \right)\right) <0, \ \ z\in[z_0,s_-(\rho_0)),
\]
which gives the claim.
\end{proof}


In the following lemma we identify a spatial scale $z_0\sim\frac1{\rho_0}$ over which we obtain quantitative lower bounds on the density $\rho_-$ over $[0,z_0]$.


\begin{lemma}\label{L:RHOMINUS}
Let $\rho_0>\frac13$ and $y_\ast\in[2,3]$ be given and consider the unique solution $(\rho_-(z;\rho_0),\omega_-(z;\rho_0))$ to the initial-value problem~\eqref{E:RHOEQN}--\eqref{E:OMEGAEQN},~\eqref{E:LEFT}.
For any $\rho_0>\frac13$ let
\begin{align}\label{E:ZNODDEF}
z_0=z_0(\rho_0): = 
\begin{cases}
\frac{\sqrt 3 }{\sqrt 2 y_\ast \rho_0} & \rho_0> 1; \\
\frac{\sqrt 3 }{\sqrt 2 y_\ast},  & \frac13 < \rho_0\le1.
\end{cases} 
\end{align}
Then $s_-(\rho_0)>z_0$ for all $\rho_0>\frac13$ and 
\begin{align}\label{E:APRIORILEFT}
\rho_-(z;\rho_0) \ge 
\begin{cases}
\rho_0 \exp\left(- \rho_0^{-1}\right), & \rho_0>1;\\
 \rho_0 \exp\left(-1\right), & \frac13 < \rho_0\le1, 
\end{cases}
\ \ z\in[0, z_0].
\end{align}
Moreover, there exists an $R>1$ such that for all $\rho_0>R$ we have 
\begin{align}\label{E:RHOZZERO}
\rho_-(z_0;\rho_0)& >\frac1{y_\ast z_0}.
\end{align}
\end{lemma}


\begin{proof}
Equation~\eqref{E:RHOEQN} is equivalent to 
\begin{align}\label{E:RHOLEFT}
\rho_-(z) = \rho_0 \exp\left(-\int_0^z\frac{2y_\ast^2 \tau \omega_-(\rho_--\omega_-)}{1- y_\ast^2 \tau^2 \omega_-^2}\,d\tau\right).
\end{align}
By Lemma \ref{L:APRIORI} we have 
the following bounds on the interval $(0, s_-(\rho_0))$
\begin{align}
\omega_-<\rho_- & <\rho_0, \label{E:B1}\\
\frac13<\omega_- & <\sqrt{\frac{\rho_0}{3}}. \label{E:B2}
\end{align}
In particular, $0<\rho_--\omega_-<\rho_0$. 
Therefore, for any $0\le z\le z_0$ using~\eqref{E:B2} 
we have 
\begin{align}
1- y_\ast^2 z^2 \omega_-^2 \ge 1 - y_\ast^2 z_0^2 \frac{\rho_0}{3} =1 - \frac1{2\rho_0}>\frac12, \label{E:B3}
\end{align} 
if $\rho_0>1$. In the case $\rho_0\in(\frac13,1]$ estimate analogous to~\eqref{E:B3} gives the same lower bound and thus $s_-(\rho_0)>z_0$ for all $\rho_0>0$.
From~\eqref{E:B1},~\eqref{E:APRIORI3},
and~\eqref{E:B3} for any $z\in [0,z_0]$ we obtain
\begin{align*}
\int_0^z\frac{2y_\ast^2 \tau \omega_-(\rho_--\omega_-)}{1- y_\ast^2 \tau^2 \omega_-^2}\,d\tau
 \le \frac{4y_\ast^2 \rho_0}{3} \int_0^z\tau\,d\tau  
  \le \frac{2y_\ast^2 \rho_0}{3}  z_0^2 
 = 
 \begin{cases}
  \rho_0^{-1}, & \rho_0>1;\\
 \rho_0 \le 1, & \frac13 < \rho_0\le1.
 \end{cases}
\end{align*}
Plugging the above bound in~\eqref{E:RHOLEFT} we obtain~\eqref{E:APRIORILEFT}.
From~\eqref{E:APRIORILEFT} the bound~\eqref{E:RHOZZERO} follows if 
$
\exp\left(-\rho_0^{-1}\right)>\frac{\sqrt 2}{\sqrt 3},
$
which is clearly true for sufficiently large $\rho_0$. 
\end{proof}

\begin{remark}\label{R:COMMONINTERVAL}
Since the mapping $\rho_0\mapsto z_0(\rho_0)$ from~\eqref{E:ZNODDEF} is nonincreasing, it follows that for any fixed $\rho_0>\frac13$ we have the uniform bound
on the sonic time:
\begin{align}
s_-(\tilde\rho_0)>z_0(\tilde\rho_0)\ge z_0(\rho_0), \ \ \text{ for all }\ \ \frac13  <\tilde\rho_0\le\rho_0.\notag
\end{align}
\end{remark}

The following lemma shows the crucial monotonicity property of $\rho_-(\cdot;\rho)$ with respect to $\rho_0$ on a time-scale of order $\sim \rho_0^{-\frac34}$.

\begin{lemma}\label{L:MONOTONE}
Let $\S\in[2,3]$. There exists a sufficiently small $\eta>0$ such that for all $\rho_0\ge \frac13$  
\begin{align}
\pa_{\rho_0}\rho_-(z;\rho_0)>0 \ \ \text{ for all } \ z\in[0, \eta \rho_0^{-\frac34}]. \notag
\end{align} 
\end{lemma}

\begin{proof}
We introduce the short-hand notation $\pa\rho_- = \pa_{\rho_0}\rho_-$ and $\pa\omega_- = \pa_{\rho_0}\omega_-$. It is easy to check that $(\pa\rho_-,\pa\omega_-)$ solve
\begin{align}
\pa\omega_-' & = - \frac{3}{z} \pa\omega_- +  
 \frac{4y_\ast^2 z\omega_-(\rho_- -\omega_-)}{(1-y_\ast^2 z^2 \omega_-^2)^2} \po - \frac{2y_\ast^2 z\omega_-^2}{1-y_\ast^2 z^2 \omega_-^2} \po   +\frac{2y_\ast^2 z \omega_-^2}{1-y_\ast^2 z^2 \omega_-^2}\pr \label{E:VARONE}\\
\pa\rho_-'  & = - \left(\frac{2y_\ast^2 z \omega_-(\rho_--\omega_-)}{1-y_\ast^2 z^2 \omega_-^2} +
\frac{2y_\ast^2 z \omega_- \rho_-}{1-y_\ast^2 z^2 \omega_-^2} \right) \pr  \notag \\
&  \ \ \ \ - \left(\frac{2y_\ast^2 z (\rho_--\omega_-)\rho_-}{1-y_\ast^2 z^2 \omega_-^2} 
- \frac{2y_\ast^2 z \omega_- \rho_-}{1-y_\ast^2 z^2 \omega_-^2} 
+ \frac{4 y_\ast^4 z^3 \omega_-^2 \rho_-(\rho_--\omega_-)}
{\left(1-y_\ast^2 z^2 \omega_-^2\right)^2}\right) \po.\label{E:VAR2}
\end{align}
At $z=0$ we have the initial values
\begin{align}\label{E:BDRYOR}
\pr(0) = 1, \ \ \po(0)=0.
\end{align}

We multiply~\eqref{E:VARONE} by $\po$ and integrate over the region $[0,z]$. By~\eqref{E:BDRYOR} we obtain
\begin{align}
\frac12\po^2(z) + \int_0^z \left(\frac3\tau +\frac{2\tau y_\ast^2 \omega_-^2}{1- y_\ast^2\tau^2 \omega_-^2 }\right) \po^2 \,d\tau
& = \int_0^z 
  \frac{4y_\ast^2 z\omega_-(\rho_- -\omega_-)}{(1-y_\ast^2 z^2 \omega_-^2)^2} \po^2\,d\tau \notag \\
& \ \ \ \  + \int_0^z \left(\frac{2y_\ast^2 \tau \omega_-^2}{1-y_\ast^2 \tau^2 \omega_-^2}\right)\pr \po\,d\tau \label{E:INT1}
\end{align}
Since $\omega_-(z)^2\le \frac{\rho_0}{3}$ (by Lemma~\ref{L:APRIORI}) and $y_\ast\le 3$ we have 
$1-y_\ast^2 \tau^2 \omega_-^2\ge 1 - 3 \tau^2\rho_0$. Therefore 
\begin{align}\label{E:LOWERBOUNDRHOZERO}
1-y_\ast^2 \tau^2 \omega_-^2 \ge \frac12, \text{ for any } \  \tau \in[0, (6\rho_0)^{-\frac12}]. 
\end{align}
Using the bounds $\ya\le3$, \eqref{E:B1}--\eqref{E:B2},~\eqref{E:LOWERBOUNDRHOZERO}, and $\rho_-\omega_-\le \frac{\rho_0}3$ (Lemma~\ref{L:APRIORI}), we obtain from~\eqref{E:INT1}
\begin{align}
& \frac12\po^2(z) + \int_0^z \left(\frac3\tau +\frac{2\tau y_\ast^2 \omega_-^2}{1- y_\ast^2\tau^2 \omega_-^2 }\right)\po^2 \,d\tau \notag \\
& \le  C \int_0^z 
\rho_0\tau \po^2\,d\tau 
 + C \rho_0 \int_0^z \tau |\pr| |\po| \, d\tau, \ \ z\le (6\rho_0)^{-\frac12} \label{E:INT2}
\end{align}
Let $Z = \eta \rho_0^{-\frac34}$ with a sufficiently small $\eta>0$ to be specified later. Note that $Z< (6\rho_0)^{-\frac12}$ for all $\rho_0\ge 1$ and $\eta$ chosen sufficiently small and independent of $\rho_0$. 
For any $\tau\in [0,Z]$ we have $\rho_0 \le \eta^{\frac43} \tau^{-\frac43}$. Therefore
$\rho_0 \tau \leq \eta^{\frac{4}3} \tau^{-\frac13}$.
From these estimates and~\eqref{E:INT2} we conclude
\begin{align}
& \frac12\po^2(z) + \int_0^z \left(\frac3\tau +\frac{2\tau y_\ast^2 \omega_-^2}{1- y_\ast^2\tau^2 \omega_-^2 }\right)\po^2 \,d\tau \notag \\
& \le  C \int_0^z 
\eta^{\frac{4}3} \tau^{-\frac13}\po^2\,d\tau 
 + \frac C{\sqrt 3} \rho_0 \left(\int_0^z\frac3\tau \po^2\,d\tau \right)^{\frac12} \left(\int_0^z \tau^3 \pr^2 \, d\tau\right)^{\frac12} \notag \\
 & \le  C \int_0^z 
 \eta^{\frac{4}3} \tau^{-\frac13}\po^2\,d\tau 
 + \frac12 \int_0^z\frac3\tau \po^2\,d\tau +  \frac {C^2}{6} \rho_0^2 \|\pr\|_{\infty}^2 \int_0^z \tau^3\,d\tau, \ \ z\in[0,Z].
 \label{E:INT3}
\end{align}
With $\eta$ chosen sufficiently small, but independent of $\rho_0$, we can absorb the first two integrals on the right-most side
into the term $\int_0^z \frac3\tau \po^2\,d\tau$ on the left-hand side. Since $\int_0^z\tau^3\,d\tau = \frac14 \eta^4\rho_0^{-3}$ we conclude
\begin{align}\label{E:POBOUND}
\lv \po(z)\rv \le C\eta^2  \rho_0^{-\frac12} \|\pr\|_{\infty}, \ \ z\in[0,Z].
\end{align}

We now integrate~\eqref{E:VAR2} and conclude from~\eqref{E:BDRYOR}
\begin{align}
\lv \pr(z) - 1\rv & \le  C\rho_0 \|\pr\|_{\infty} \int_0^z \tau\,d\tau   
+  C\|\po\|_{\infty} \int_0^z \left(\rho_0^2 \tau + \rho_0 \tau + \rho_0^{2}\tau^3\right) \,d\tau \notag \\
& \le C\eta^2  \|\pr\|_{\infty},  \ \ z\in[0,Z],  \notag
\end{align}
where we have used~\eqref{E:POBOUND},~\eqref{E:APRIORI3}, and $0\le z\le \eta \rho_0^{-\frac34}$. Therefore, 
\[
\|\pr\|_{\infty} \le 1 + C\eta^2 \|\pr\|_{\infty}
\]
and thus, for $\eta$ sufficiently small so that $C\eta^2<\frac13$, we have $\|\pr\|_\infty\le \frac32$. From here we infer
\begin{align}
\pr(z) \ge 1 - \frac32C\eta^2>\frac12>0, \ \ z\in[0,Z].\notag
\end{align}
\end{proof}

\subsection{Existence of the LP-solution connecting the origin to the sonic point} \label{S:INTERSECTION}

The goal of this section is to carry out the intersection argument to show that $\lim_{z\to0^+}\omega(z;\bar y_\ast)=\frac13$. 
Before that we prove an important technical lemma that will be used later on.


\begin{lemma}\label{L:XIMP}
Let $x_\ast\in \X$ (see~\eqref{E:XDEF} for the definition of $\X$) and assume that $s(x_\ast)=0$. 
Then
\begin{enumerate}
\item[(a)]
\begin{align}
\rho(z;x_\ast)>\omega(z;x_\ast), \ \ z\in(0,1);\notag
\end{align}
\item[(b)]
\begin{align}
\limsup_{z\to0} z \omega(z;x_\ast) >0.\notag
\end{align}
\end{enumerate}
\end{lemma}


\begin{proof}
{\em Proof of Part (a).}
If not let
\begin{align}
z_c : = \sup_{z\in(0,1)} \left\{\rho(\tau;x_\ast)-\omega(\tau;x_\ast)>0, \ \tau\in(z,1) , \ \ \rho(z;x_\ast)=\omega(z;x_\ast)\ \right\}>0.
\notag
\end{align}  
At $z_c$ we have from~\eqref{E:RHOEQN}--\eqref{E:OMEGAEQN} $\omega'(z_c;x_\ast) = \frac{1-3\omega}{z_c}<0$ and $\rho'(z_c;x_\ast)=0$.  Therefore there exists a neighbourhood strictly to the right of $z_c$ such that $\omega'<0$, $\rho<\omega$, and $\rho'>0$. It is easily checkes that this property is  dynamically trapped and we conclude 
\begin{align}
\omega'(z;x_\ast) \le \frac{1-3\omega(z;x_\ast)}{z} , \ \ z\le z_c. \label{E:OMEGAPRIME}
\end{align}
Integrating the above equation over $[z,z_c]$ we conclude 
\begin{align}
\omega(z;x_\ast) z^3 \ge \omega(z_c;x_\ast) z_c^3 - \frac13 z_c^3 = \left(\omega(z_c;x_\ast)-\frac13\right) z_c^3 =:c>0.\notag
\end{align}
In other words $\omega(z;x_\ast)z\ge \frac c{z^2}\gg 1$ for sufficiently small $z$, which implies  $s(x_\ast)>0$. A contradiction. 

\noindent
{\em Proof of Part (b).}
By way of contradiction we assume that $\lim_{z\to0}z\omega(z;x_\ast)=0$. For any $\epsilon>0$ choose $\delta>0$ so small that 
\[
1- x_\ast^2 z^2 \omega(z;x_\ast)^2>1-\epsilon^2, \ \ \text{ i. e. }z x_\ast \omega(z;x_\ast)<\epsilon, \ \ z\in(0,\delta).
\]
From~\eqref{E:OMEGAEQN} and \eqref{E:RHOAPRIORI} we then conclude
\begin{align}
\omega' & \le \frac{1-3\omega}z  + \frac{2 \epsilon x_\ast \omega }{1-\epsilon^2} \frac1{x_\ast z} =
  \frac{1-\left(3-C_\ast\epsilon\right)\omega}{z}, \label{E:OMEGAPRIME}
\end{align}
where $C_\ast =\frac2{1-\epsilon^2}$. Letting $i_\ast:=\inf_{z\in(0,1]}\omega(z;x_\ast)>\frac13$, we choose $\epsilon>0$ so small that 
\begin{align}
1-\left(3-C_\ast\epsilon\right)\omega(z;x_\ast) < 1-\left(3-C_\ast\epsilon\right) i_\ast < - c_\ast<0, \ \ c_\ast : = -\frac{1-3i_\ast}2.\notag
\end{align}
From~\eqref{E:OMEGAPRIME}
\begin{align}\label{E:OMEGAPRIME2}
\omega' \le -\frac{c_\ast}{z}, \ \ z\in(0,\delta). 
\end{align}
Therefore 
\[
\omega(z;x_\ast) = \omega (\delta;x_\ast) - \int_{z}^\delta \omega'(\tau;x_\ast)\,d\tau \ge  \omega (\delta;x_\ast) +c_\ast \log\frac{\delta}{z} \longrightarrow_{z\to 0} \infty.
\]
As a consequence of~\eqref{E:OMEGAPRIME}, for sufficiently small $\epsilon$ and $z\ll1$ we have
\begin{align}
\omega' \le -\frac{2\omega}{z}, \notag
\end{align}
which in turn implies $\omega(z;x_\ast)\ge C z^{-2}$ for sufficiently small $z$. Therefore, since $\omega(z;x_\ast)z >1$ for sufficiently small $z$ we conclude $s(\bar x_\ast)>0$, a contradiction to the assumption $s(x_\ast)=0$.
\end{proof}


We now recall Definition~\ref{D:ULS}, where the notion of an upper and a lower solution is introduced. The next lemma shows that we can find a lower solution at a point $0<z_0\ll1$ arbitrarily close to $z=0$.

\begin{lemma}[Existence of a lower solution]\label{L:LOWERSOLUTION}
There exists an $\eta>0$ such that for any $z_0<\eta$ there exists an $y_{\ast\ast}\in [\bar y_\ast,3]$ such that $(\rho(\cdot;y_{\ast\ast}),\omega(\cdot;y_{\ast\ast}))$ is a lower solution at $z_0$. Moreover, there exists a universal constant $C$ such that $\rho_1<\frac C{z_0}$, where $\rho_-(z_0;\rho_1)=\rho(z_0;y_{\ast\ast})$.
\end{lemma}


\begin{proof}
For any $y_\ast\in Y$ we consider the function 
\begin{align}\label{E:BIGFDEF}
F(y_\ast) : = \sup_{\tilde y_\ast\in[\bar y_\ast, y_\ast]} \left\{z_{\frac13}(\tilde y_\ast)\right\}. 
\end{align}
The function $y_\ast\mapsto F(y_\ast)$ is clearly increasing, continuous,  and by Lemma~\ref{L:ZONETHIRDCONTINUITY} $\lim_{y_\ast\to\bar y_\ast}F(y_\ast)=0$.  Therefore, the range of $F$ is of the form $[0,\eta]$ for some $\eta>0$.
For any $y_\ast\in Y$, by Lemma~\ref{L:ZONETHIRDCONTINUITY}, the supremum in~\eqref{E:BIGFDEF} is attained, i.e. there exists $y_{\ast\ast}\in [\bar y_\ast, y_\ast]$ such that $F(y_\ast) = z_{\frac13}(y_{\ast\ast})=:z_0$. Therefore, for any 
$\bar y_\ast < y_\ast< y_{\ast\ast}$ we have
\begin{align}
s(y_\ast)<z_{\frac13}(y_\ast)\le z_0.\notag
\end{align}
Due to~ \eqref{E:WLESSTHANRZ} we have the bound $\rho(z_0;y_{\ast\ast})>\omega(z_0;y_{\ast\ast})=\frac13$. By Lemma~\ref{L:RHOMINUS} choosing $\rho_0 = \rho_0(z_0) = \frac{\sqrt 3}{\sqrt 2 y_{\ast\ast} z_0}>1$ 
we have
\begin{align}
\rho_-(z_0;\rho_0) >\frac1{y_{\ast\ast}z_0}>\rho(z_0;y_{\ast\ast}),\notag
\end{align}
where we have used Lemma~\ref{L:APRIORIODE} in the last bound.
On the other hand $\rho_-(z_0;\frac13)=\frac13<\rho(z_0;y_{\ast\ast})$ (where we recall that $\rho_-(\cdot;\frac 13)$ is the Friedman solution, see Remark~\ref{R:FRIEDMAN}). Using Remark~\ref{R:COMMONINTERVAL} and the Intermediate Value Theorem, there exists a $\rho_1\in(\frac13,\rho_0)$ such that 
\begin{align}
\rho(z_0;y_{\ast\ast}) = \rho_-(z_0;\rho_1).\notag
\end{align}
By~\eqref{E:APRIORI1} $\omega_-(z_0;\rho_1)>\frac13 = \omega(z_0;y_{\ast\ast})$ and therefore $(\rho(\cdot,y_{\ast\ast}),\omega(\cdot;y_{\ast\ast}))$ is a lower solution at $z_0$. The upper bound on $\rho_1$ follows from our choice of $\rho_0$.
\end{proof}


The most delicate argument in this section is the following lemma, which states that $(\rho(\cdot;\S), \omega(\cdot;\S))$ is an upper solution at some $z_0\ll1$ if $\lim_{z\to0}\omega(z;\S)\neq\frac13$.

\begin{lemma}\label{L:CONTR}
If 
\begin{align}
\lim_{z\to0} \omega(z;\bar y_\ast) \neq \frac13,\notag
\end{align}
then there exists  a universal constant $C$ and an arbitrarily small  $z_0>0$ such that $(\rho(\cdot;\bar y_\ast),\omega(\cdot;\bar y_\ast))$ is an upper solution at $z_0$ and $\rho_1<\frac C{z_0}$, where $\rho_-(z_0;\rho_1)=\rho(z_0;\bar y_\ast)$.
\end{lemma}

\begin{proof}
It is clear that $\liminf_{z\to0}\omega(z;\bar y_\ast)\ge\frac13$ as otherwise we would have $\bar y_\ast\in Y$, a contradiction to the definition~\eqref{E:BARYDEF} of $\bar y_\ast$ and the openness of $Y$.
We distinguish three cases.

\noindent
{\em Case 1.}
\begin{align}
\liminf_{z\to 0} \omega(z;\bar y_\ast) >\frac13.\notag
\end{align}
In this case $\bar y_\ast\in \X$ and by Proposition~\ref{P:EXISTENCEY} we have $s(\bar y_\ast)=0$.
By part (b) of Lemma~\ref{L:XIMP} there exists a constant $C>0$ and a sequence $\{z_n\}_{n\in\mathbb N}\subset(0,1)$ such that $\lim_{n\to\infty}z_n=0$ and  
\be\label{E:CRUCIAL00}
\omega(z_n;\bar y_\ast)>\frac C{\bar y_\ast z_n}.
\ee
For any such $z_n$ we have by part (a) of Lemma~\ref{L:XIMP} and Lemma~\ref{L:APRIORIODE}
\begin{align}\label{E:CONSTANTC}
\rho(z_n;\bar y_\ast)>\omega(z_n;\bar y_\ast) >\frac{C}{\bar y_\ast z_n} > C \rho(z_n;
\bar y_\ast).
\end{align}

For any $0<z_n\ll1$ sufficiently small consider $(\rho_-(\cdot;\rho^n_0),\omega_-(\cdot;\rho^n_0))$ with $\rho^n_0=\rho_0(z_n)= \frac{\sqrt 3 }{\sqrt 2 \bar y_\ast z_n}>1$. By Lemmas~\ref{L:RHOMINUS} and~\ref{L:XIMP}
\be
\rho_-(z_n;\rho^n_0)> \frac1{\bar y_\ast z_n}>\rho(z_n;\bar y_\ast)>\omega(z_n;\bar y_\ast)>\frac13.\notag 
\ee
On the other hand,
\begin{align}
\rho(z_n;\bar y_\ast)>\frac13 = \rho_-(z_n;\frac13),\notag
\end{align}
where we recall that $\rho_-(\cdot;\frac13)\equiv\frac13$ is the Friedman solution.
Moreover, by Remark~\ref{R:COMMONINTERVAL} $[0,z_n]\subset[0,s_-(\tilde \rho_0))$ for all $\tilde \rho_0\subset[\frac13,\rho^n_0]$. By the continuity of the map 
$[\frac13,\rho^n_0]\ni \tilde \rho_0\mapsto \rho_-(z_n;\tilde \rho_0)$  the Intermediate Value Theorem implies that there exists  $\rho^n_1\in(\frac13,\rho^n_0)$ such that 
\begin{align}\label{E:EQUALITY}
\rho_-(z_n;\rho_1^n) = \rho(z_n;\bar y_\ast) \ \ \text{ for all sufficiently large $n\in\mathbb N$}.
\end{align}
Let 
\begin{align}
c^n_1: =
\begin{cases}
 \exp\left(- \left(\rho^n_1\right)^{-1}\right), & \text{ if }\ \rho^n_1> 1; \\
 \exp\left(-1 \right),  &  \text{ if }\ \frac13 <\rho^n_1 \le1. \notag
\end{cases}
\end{align}
Clearly $c_1^n\ge e^{-1}=:c_1$ for all $n\in\mathbb N$.
By Lemma~\ref{L:RHOMINUS} $\rho_-(z_n;\rho^n_1)\ge c_1\rho^n_1$. Since $\omega_-(z_n;\rho^n_1)< \left(\frac{\rho^n_1}{3}\right)^{\frac12}$, we conclude together with~\eqref{E:CONSTANTC} and~\eqref{E:EQUALITY} that for all $n$ sufficiently large
\begin{align}\label{E:UPPERONE}
\omega_-(z_n;\rho^n_1)<\frac{1}{\sqrt{3c_1}}\rho_-(z_n;\rho^n_1)^{\frac12} = \frac{1}{\sqrt{3c_1}}\rho(z_n;\bar y_\ast)^{\frac12} 
\le \frac{1}{\sqrt{3c_1 C}} \omega(z_n;\bar y_\ast)^{\frac12}. 
\end{align}
By~\eqref{E:CONSTANTC} $\omega(z_n;\bar y_\ast)$ grows to positive infinity as $z_n$ approaches zero. Therefore, we may choose a sufficiently large $N\in\mathbb N$ and set $z_0=z_N\ll1$, $\rho_0=\rho^N_0$, $\rho_1=\rho_1^N$ so that 
$ \frac{1}{\sqrt{3c_1 C}} \omega(z_0;\bar y_\ast)^{\frac12}<\omega(z_0;\bar y_\ast)$. Together with~\eqref{E:UPPERONE} this gives
\begin{align}
\omega_-(z_0;\rho_1) < \omega(z_0;\bar y_\ast).\notag
\end{align}
We conclude that $(\rho(\cdot;\bar y_\ast), \omega(\cdot;\bar y_\ast))$ is an upper solution (see Definition~\ref{D:ULS}) at $z_0$ and the upper bound on $\rho_1$ follows from our choice of $\rho_0$.


\noindent
{\em Case 2.}
\begin{align}\label{E:CONTR2}
\frac13<\limsup_{z\to 0} \omega(z;\bar y_\ast) <\infty, \ \ \liminf_{z\to 0}\omega(z;\bar y_\ast)=\frac13.
\end{align}
In particular $\bar y_\ast\in \Z$ (see~\eqref{E:ZDEF}) and by Lemma~\ref{L:PREPZ} 
$\rho(z;\bar y_\ast)>\omega(z;\bar y_\ast)$.
Assumption~\eqref{E:CONTR2}  also implies that there exists a constant $c>0$ independent of $z$ such that 
\begin{align}\label{E:OMEGAUNIFORMBOUND}
\omega(z;\bar y_\ast)<c, \ \ z\in(0,1].
\end{align}
From~\eqref{E:RHOEQN} and the bound~\eqref{E:RHOAPRIORI} 
we conclude
\begin{align}
\rho'(z;\bar y_\ast) \ge - C \rho(z;\bar y_\ast), \notag 
\end{align}
or equivalently $\left(\rho e^{Cz}\right)'\ge0$; here $C>0$.  This implies the boundedness of $\rho(\cdot;\bar y_\ast)$, i.e. 
\begin{align}\label{E:RHOUNIFORMBOUND}
\rho(z;\bar y_\ast)<c, \ \ z\in(0,1],
\end{align}
where we have (possibly) enlarged $c$ so that~\eqref{E:OMEGAUNIFORMBOUND} and~\eqref{E:RHOUNIFORMBOUND} are both true. There exists an $\eta>0$ and a sequence $\left\{z_n\right\}_{n\in\mathbb N}$ such that $\lim_{n\to\infty}z_n=0$
and 
\[
\frac13 + \eta < \omega(z_n;\bar y_\ast), \ \ \text{ and } \ \ \lim_{n\to \infty} \omega(z_n;\bar y_\ast) = \limsup_{z\to 0} \omega(z;\bar y_\ast).
\] 
Since $\{ \rho(z_n;\bar y_\ast)\}_{n\in\mathbb N}$ is bounded, by Lemma~\ref{L:RHOMINUS} we can choose a $\rho_0>1$ such that 
$
\rho_-(z_n;\rho_0) > \rho(z_n;\bar y_\ast)
$
for all $n\in\mathbb N$. 
On the other hand $\rho(z_n;\bar y_\ast)>\frac13=\rho_-(z_n;\frac13)$. By the intermediate value theorem there exists a 
sequence $\{\rho_0^n\}_{n\in\mathbb N}\subset (\frac13,\rho_0)$ such that
\begin{align}
\rho_-(z_n;\rho_0^n) = \rho(z_n;\bar y_\ast).\notag
\end{align}
Since $\omega_-(z;\rho_0^n)^2\le \frac{\rho_0^n}{3}<\frac{\rho_0}{3}$ and $\rho_-(z;\rho_0^n)<\rho_0^n<\rho_0$ (Lemma~\ref{L:APRIORI}) we conclude from~\eqref{E:RHOEQN}--\eqref{E:OMEGAEQN} and Theorem~\ref{T:ANALYTICITYLEFT} that $\lv \rho_-'(z_n;\rho_0^n) \rv$ and $\lv \omega_-'(z_n;\rho_0^n)\rv$ are bounded uniformly-in-$n$, by some constant, say $C$. 
Therefore
\begin{align}
\omega_-(z_n;\rho_0^n) \le \frac13 + C z_n.\notag
\end{align}
We thus conclude that for a fixed $n$ sufficiently large $\omega_-(z_n;\rho_0^n)<\frac13 + \eta <\omega(z_n;\bar y_\ast)$. Therefore, $\omega(\cdot;\bar y_\ast)$ is an upper solution (see Definition~\ref{D:ULS}) at $z_0:=z_n$ with $\rho_1=\rho_0^n$.
The claimed upper bound on $\rho_1$ is clear.

\noindent
{\em Case 3.}
\begin{align}
\frac13<\limsup_{z\to 0} \omega(z;\bar y_\ast) =\infty, \ \ \liminf_{z\to0}\omega(z;\bar y_\ast)=\frac13.\notag
\end{align}
As $\omega(\cdot;\bar y_\ast)$ must oscillate between $\frac13$ and $\infty$ we can use the mean value theorem to conclude that there exists a sequence $\left\{z_n\right\}_{n\in\mathbb N}$ such that $\lim_{n\to\infty}z_n=0$ and
\begin{align}\label{E:CRUCIAL0}
\omega(z_n;\bar y_\ast)>n, \ \ \text{ and } \ \ \omega'(z_n;\bar y_\ast)=0.
\end{align}
We claim that 
there exist $N_0>0$ and $0<\eta\ll1$ such that
\begin{align}\label{E:CRUCIAL}
\omega(z_n;\bar y_\ast)\ge  \frac{r-\eta}{\bar y_\ast z_n}, \ \ n\ge N_0,
\end{align}
where $r<1$ is the positive root of the quadratic polynomial $3x^2+2x-3$.  To prove this, assume that~\eqref{E:CRUCIAL} is not true. 
Then there exists a subsequence of $\{z_n\}_{n\in\mathbb N}$ such that $\omega(z_n;\bar y_\ast)< \frac{r-\eta}{\bar y_\ast z_n}$ and therefore
\begin{align}
\frac{2\bar y_\ast z_n\omega(z_n;\bar y_\ast)}{1 - (\bar y_\ast z_n\omega(z_n;\bar y_\ast))^2} < 3 - C(\eta),\notag
\end{align}
for some $C(\eta)>0$.
Since $\rho-\omega<\frac1{\bar y_\ast z}$ by~\eqref{E:RHOAPRIORI} we have
\begin{align}
\omega'(z_n;\bar y_\ast)& <\frac{1-3\omega(z_n;\bar y_\ast)}{z_n} + \frac{2\bar y_\ast z_n\omega(z_n;\bar y_\ast)}{1 - (\bar y_\ast z_n\omega(z_n;\bar y_\ast))^2}\frac{\omega(z_n;\bar y_\ast)}{z_n} \notag \\
& = \frac{\omega(z_n;\bar y_\ast)} {z_n}\left(\frac1{\omega(z_n;\bar y_\ast)}-3 +  \frac{2\bar y_\ast z_n\omega(z_n;\bar y_\ast)}{1 - (\bar y_\ast z_n\omega(z_n;\bar y_\ast))^2} \right) \notag \\
& <  (\frac1{\omega(z_n;\bar y_\ast)} -C(\eta)) \frac{\omega(z;\bar y_\ast)}{z} <0 \ \ \text{ for $n$ sufficiently large}.\notag
\end{align}
This is a contradiction to~\eqref{E:CRUCIAL0}. 
We can therefore repeat the same argument following~\eqref{E:CRUCIAL00} to conclude that $\omega(\cdot;\bar y_\ast)$ is an upper solution at $z_0:=z_n$, for some $n$ sufficiently large. The upper bound on $\rho_1$ follows in the same way.
\end{proof}

We now use a continuity argument to show the following key proposition.

\begin{proposition}\label{P:ONETHIRD}
The limit $\lim_{z\to0}\omega(z;\bar y_\ast)$ exists and
\be
\lim_{z\to0}\omega(z;\bar y_\ast) = \frac13.\notag
\ee
\end{proposition}


\begin{proof}
Assume that the claim is  not true. By Lemmas~\ref{L:LOWERSOLUTION} and~\ref{L:CONTR} we can find a $0<z_0\ll1$ and $y_{\ast\ast}\in Y$ so that $(\rho(\cdot;y_{\ast\ast}),\omega(\cdot;y_{\ast\ast}))$  and $(\rho(\cdot;\bar y_\ast),\omega(\cdot;\bar y_\ast))$   are respectively a lower and an upper solution at $z_0$.
Without loss of generality let 
\[
A: =\rho(z_0;y_{\ast\ast})< \rho(z_0;\bar y_\ast)=:B.
\]
By Lemmas~\ref{L:LOWERSOLUTION} and~\ref{L:CONTR} there exist $\rho_A,\rho_B>\frac13$ such that $A=\rho_-(z_0;\rho_A)$, $B=\rho_-(z_0;\rho_B)$, and
$\rho_A,\rho_B\in(\frac13,\rho_0)$, where $\rho_0\gg1$ and $z_0 \le C\frac1{\rho_0}$. Therefore by Lemma~\ref{L:MONOTONE}, $\pa_{\rho_0}\rho_-(z_0;\tilde\rho_0)>0$ for all $\tilde\rho_0\in[\frac13,\rho_0]$, since $\rho_0^{-\frac34}\gg \rho_0^{-1}$ for $\rho_0$ large. By the inverse function theorem, there exists a continuous function $\tau\mapsto f(\tau)$ such that 
\begin{align}
\rho_-(z_0;f(\tau)) & = \tau,  \ \ \tau\in[A,B] \notag\\
f(\rho_A) & = A.\notag
\end{align} 
By strict monotonicity of $\tilde\rho_0\mapsto \rho_-(z_0;\tilde\rho_0)$ on $(0,\rho_0]$ the inverse $f$ is in fact injective and therefore $f(\rho_B)=B$. 
We consider the map 
\begin{align}
[\bar y_\ast,y_{\ast\ast}] \ni y_\ast \mapsto  \omega(z_0;y_\ast) - \omega_-(z_0;f(\rho(z_0;y_\ast))) = : h(y_\ast).\notag
\end{align}
By the above discussion $h$ is continuous, $h(A)<0$ and $h(B)>0$. Therefore, by the Intermediate Value Theorem there exists a $y_s\in(\bar y_\ast,y_{\ast\ast})$ such that $h(y_s)=0$.  The solution $(\rho(\cdot;y_s),\omega(\cdot;y_s))$ exists on $[0,1]$, satisfies $\omega(0)=\frac13$ and belongs to $Y$. This is a contradiction to~\eqref{E:WLESSTHANTHIRDZ}.
\end{proof}


It remains to show that the solution is regular at $z=0$ and we do this by showing that it coincides with a solution $(\rho_-(\cdot;\rho_\ast),\omega_-(\cdot;\rho_\ast))$ emanating from the origin, with the correct choice of $\rho_\ast$.

\begin{proposition}\label{P:GOODPROPERTIES}
There exists a constant $C_\ast>0$ so that
\begin{align}
\lv \rho(z;\bar y_\ast)\rv + \lv \omega(z;\bar y_\ast)\rv
 +  \lv \frac{\omega(z;\bar y_\ast)-\frac13}{z^2}\rv
\le C_\ast, \ \ z\in(0,1].\notag
\end{align} 
The solution $\rho(\cdot;\bar y_\ast):(0,1]\to\mathbb R_{>0}$ extends continuously to $z=0$ and 
$\rho_\ast:=\rho(0;\bar y_\ast)<\infty$. Moreover, the solution $(\rho(\cdot;\bar y_\ast),\omega(\cdot;\bar y_\ast))$ coincides with
$(\rho_-(\cdot;\rho_\ast),\omega_-(\cdot;\rho_\ast))$ and it is therefore analytic at $z=0$ by Theorem~\ref{T:ANALYTICITYLEFT}.
\end{proposition}

\begin{proof}
By Proposition~\ref{P:ONETHIRD} it is clear that $\omega(\cdot;\bar y_\ast)$ is bounded on $[0,1]$. From~\eqref{E:RHOEQN} and~\eqref{E:RHOAPRIORI} we conclude $|\rho'| \le C\rho$ and thus $\rho$ is  bounded up to $z=0$. Using the boundedness of $\rho$ and $\omega$, equation~\eqref{E:RHOEQN} immediately implies $|\rho'(z)|\lesssim z$ for all $z\in[0,1]$.

Since $\rho(\cdot;\bar y_\ast)-\omega(\cdot;\bar y_\ast)\ge0$ on $(0,1]$ and  both $\rho(\cdot;\bar y_\ast)$ and $\omega(\cdot;\bar y_\ast)$ are positive, we conclude from~\eqref{E:RHOEQN} that $\rho'\le0$ and therefore the limit $\rho_\ast:=\lim_{z\to0}\rho(z;\bar y_\ast)$ exists and by the above it is finite.
Let
\[
\zeta = \omega-\frac13\ge0.
\]
It is then easy to check from~\eqref{E:OMEGAEQN}
\begin{align}
\left(\zeta z^3\right)' = z^3 \frac{2 \bar y_\ast^2 z \omega^2(\rho-\omega)}{1-\bar y_\ast^2 z^2\omega^2}\notag
\end{align}
and therefore, since $\omega$ and $\rho$ are uniformly bounded for any $0<z_1<z$ we obtain
\[
\zeta(z) z^3 - \zeta(z_1) z_1^3 \le C \int_{z_1}^{z}\tau^4\,d\tau  = \frac C5\left(z^5-z_1^5\right) .
\]
We now let $z_1\to0^+$ and conclude 
\be\label{E:ZETABOUND}
\zeta(z)\lesssim z^2
\ee

Consider now $\bar\rho(z):=\rho(z;\bar y_\ast)-\rho_-(z;\rho_\ast)$ and $\bar\omega(z):=\omega(z;\bar y_\ast)-\omega_-(z;\rho_\ast)$, both are defined in a (right) neighbourhood of $z=0$ and satisfy
\begin{align}
\bar \rho' &= O(1)\bar\rho + O(1)\bar\omega \notag \\
\bar\omega' & = - 3\frac{\bar\omega}{z} + O(1)\bar\rho + O(1)\bar\omega, \notag
\end{align}
where $\bar\rho(0)=\bar\omega(0)=0$. Here we have used the already proven boundedness of $(\rho(\cdot;\bar y_\ast),\omega(\cdot;\bar y_\ast))$ and the boundedness of $(\rho(\cdot;\rho_\ast),\omega(\cdot;\rho_\ast))$, see Lemma~\ref{L:APRIORI}.
We multiply the first equation by $\bar\rho$, the second by $\bar\omega$, integrate over $[0,z]$ and use Cauchy-Schwarz to get
\begin{align}\label{E:UNIQUE}
\bar\rho(z)^2+\bar\omega(z)^2 + 3\int_0^z\frac{\bar\omega(\tau)^2}{\tau}\,d\tau  \le C \int_0^z \left(\bar\rho(\tau)^2+\bar\omega(\tau)^2\right)\,d\tau.
\end{align}
We note that $\int_0^z\frac{\bar\omega(\tau)^2}{\tau}\,d\tau $ is well-defined, since $\bar\omega = \zeta - \zeta_-$, where $\zeta_-=\omega_--\frac13$; we use~\eqref{E:ZETABOUND} and observe $\zeta_-\lesssim z^2$ in the vicinity of $z=0$ by the analyticity of $\omega_-$, see Theorem~\ref{T:ANALYTICITYLEFT}. Therefore $\bar\rho(z)^2+\bar\omega(z)^2=0$ by~\eqref{E:UNIQUE}. The analyticity claim now follows from Theorem~\ref{T:ANALYTICITYLEFT}.
\end{proof}

\section{Proof of the main theorem}\label{S:MAIN}

The existence of an LP-type solution, the corresponding boundary conditions stated in Theorem~\ref{T:MAIN}, and~\eqref{E:DENSITYPROP} follow directly from Propositions~~\ref{P:FGE},~\ref{P:EXISTENCEY},~\ref{P:GOODPROPERTIES}. The real analyticity locally around the sonic point and the origin follows from Theorems~\ref{T:ANALYTICITY} and~\ref{T:ANALYTICITYLEFT} respectively, while away from these two points it follows from the standard ODE-theory. Unwinding the change of variables~\eqref{E:ZVARIABLE} and~\eqref{E:RELVEL} we easily obtain~\eqref{E:VELPROP} from~\eqref{E:OMEGAUPDOWN} when $z\ge 1$ or equivalently $y\ge \bar y_\ast$. To get~\eqref{E:VELPROP} for $z<1$ we first rewrite~\eqref{E:OMEGAEQN} in the form
\[
\omega' = \frac{1-\omega}z + \frac{-2\omega+2\bar y_\ast^2z^2\omega^2\rho}{z\left(1-\bar y_\ast^2z^2\omega^2\right)}.
\]
For any $z\in(0,1)$ we have $\frac1{\bar y_\ast^2z^2}\omega(z;\bar y_\ast)\rho(z;\bar y_\ast)<1$ by~\eqref{E:RHOAPRIORI} and the bound $\omega<\rho$ which follows from~\eqref{E:WLESSTHANRZ} and  $\bar y_\ast\in\Z$. This yields
\be\label{E:OMEGABOUNDBYONE}
\omega'<\frac{1-\omega}z, \ \ z\in(0,1).
\ee
Let $z_c:=\sup\{z\in(0,1)\,\big| \omega(z;\bar y_\ast)<1\}$. We use a contradiction argument and assume $z_c<1$. Since $\omega(0;\bar y_\ast)=\frac13$ it follows by continuity that $z_c>0$. At $z_c$ we must therefore have $\omega(z_c)=1$ and $\omega'(z_c)\ge0$. On the other hand, from~\eqref{E:OMEGABOUNDBYONE} we conclude
\[
\omega'(z_c;\bar y_\ast)<\frac{1-\omega(z_c;\bar y_\ast)}{z_c} =0,
\]
a contradiction. 
We conclude that for any $z\in[0,1)$ we have $\omega(z)<1$. By Proposition~\ref{P:EXISTENCEY} we also have $\omega(z,\bar y_\ast)\ge\frac13$ and therefore~\eqref{E:VELPROP} follows also in the region $z\in[0,1)$ or equivalently $y\in[0,\bar y_\ast)$.

\medskip

{\bf Acknowledgments.}
Y. Guo's research is supported in part by NSF DMS-grant 1810868.
M. Had\v zi\'c's research is supported by the EPSRC Early Career Fellowship EP/S02218X/1.
J. Jang's research is supported by the NSF DMS-grant 2009458
and the Simons Fellowship (grant number 616364).


\begin{thebibliography}{99}


\bibitem{BaChFe2009}
Bae, M., Chen, G.-Q., Feldman, M., 
Regularity of solutions to regular shock reflection for potential flow. 
{\em Invent. Math.} 
{\bf 175}, no. 3, 505--543 (2009).

\bibitem{BrWi1998}
Brenner, M. P.,  Witelski, T. P.,  
On Spherically Symmetric Gravitational Collapse.
{\em J. Stat. Phys.}
{\bf 93}, 3/4, 863--899 (1998).

\bibitem{ChFe2018}
Chen, G.-Q., Feldman, M.,  
The mathematics of shock reflection-diffraction and von Neumann's conjectures. 
{\em Annals of Mathematics Studies}, {\bf 197} 
Princeton University Press, Princeton, NJ, 2018.

\bibitem{DLYY} 
Deng, Y.,  Liu, T.P., Yang, T.,  Yao Z.,
Solutions of Euler-Poisson equations for gaseous stars.  
\emph{Arch. Ration. Mech. Anal.} \textbf{164}, no. 3, 261--285  (2002).

\bibitem{FuLin} 
Fu, C.-C.; Lin, S.-S.,
On the critical mass of the collapse of a gaseous star in spherically symmetric and isentropic motion. 
{\em Japan J. Indust. Appl. Math.},
\textbf{15}, no. 3, 461--469 (1998).

\bibitem{GeMaPa}
Germain, P., Masmoudi, N., Pausader, B.,
Non-neutral global solutions for the electron Euler-Poisson system in 3D,
{\em SIAM J. Math. Anal.}, {\bf 45-1}, 267--278 (2013).

\bibitem{GoWe}
Goldreich, P.,  Weber,  S., 
Homologously collapsing stellar cores. 
\emph{Astrophys. J.}, 
\textbf{238},  991--997  (1980).

\bibitem{Guo1998}
Guo, Y., 
Smooth irrotational Flows in the large to the Euler-Poisson system in $\mathbb R^{3+1}$, 
{\em Commun. Math. Phys.} 
{\bf 195}, 249--265 (1998).

\bibitem{GHJ}
Guo, Y., Had\v zi\'c, M., Jang, J.,
Continued Gravitational Collapse for Newtonian Stars,
{\em Arch. Rat. Mech. Anal.}, https://doi.org/10.1007/s00205-020-01580-w (2020).

\bibitem{GuPa}
Guo, Y., Pausader, B.,
Global Smooth Ion Dynamics in the Euler-Poisson System.
{\em Comm. Math. Phys.} 
{\bf 303}, 89-125, (2011).

\bibitem{GaGu2007}
Gundlach, C., Mart\'in-Garc\'ia, J.~M.,
Critical Phenomena in Gravitational Collapse.
{\em Living Reviews in Relativity}, 2007.

\bibitem{Hartmann2009}
L. Hartmann,
{\em Accretion Processes in Star Formation.}
Cambridge University Press, 2nd edition (2009)

\bibitem{HaJa2016-1}
 Had\v zi\'c, M., Jang, J.,
Nonlinear stability of expanding star solutions in the radially-symmetric mass-critical Euler-Poisson system.
{\em Comm. Pure Appl. Math.}, 
{\bf 71}, no. 5,  827--891 (2018).


\bibitem{Hunter1977}
Hunter, C., 
The collapse of unstable isothermal spheres. 
{\em Astrophys. J.} 
{\bf 218}, 
834--845 (1977).

\bibitem{IoPa2013}
Ionescu, A., Pausader, B.,
The Euler-Poisson system in 2D: global stability of the constant equilibrium solution,
{\em Int. Math. Res. Not.}, 761--826 (2013).

\bibitem{JeTs}
Jensen, H. K., Tsikkou, C.,
Multi-d isothermal Euler flow: Existence of unbounded radial similarity solutions
{\em Physica D: Nonlinear Phenomena}, 
{\bf 410}, 132511 (2020)

\bibitem{Larson1969}
Larson, R. B., 
Numerical calculations of the dynamics of a collapsing protostar. 
{\em Mon. Not. R. Astr. Soc.} 
{\bf 145}, 271--295 (1969).

\bibitem{MaedaHarada2001}
Maeda, H., Harada, T.,
Critical phenomena in Newtonian gravity
{\em Phys. Rev. D} 
{\bf 64}, 124024 (2001).

\bibitem{Makino92}
Makino, T.,
Blowing up solutions of the Euler-Poisson equation for the evolution of gaseous stars.
{\em Transport Theory Statist. Phys.} \textbf{21}, 615--624 (1992).

\bibitem{MRRS}
Merle, F., Rapha\"el, P., Rodnianski, I., Szeftel, J.,
On smooth self similar solutions to the compressible Euler equations.
{\em Preprint, available online at https://arxiv.org/abs/1912.10998} (2019).

\bibitem{MRRS2}
Merle, F., Rapha\"el, P., Rodnianski, I., Szeftel, J.,
On the implosion of a three dimensional compressible fluid.
{\em Preprint, available online at https://arxiv.org/abs/1912.11009} (2019).

\bibitem{Penston1969}
Penston, M. V., 
Dynamics of self-gravitating gaseous spheres III, 
{\em Mon. Not. R. Astr. Soc.} 
{\bf 144},
425--448 (1969).

\bibitem{OriPiran1988}
Ori A.,  Piran, T.,  
A simple stability criterion for isothermal spherical self-similar flow.
{\em Mon. Not. R. Astr. Soc.} 
{\bf 234}, 821--829 (1988).

\bibitem{Shu1977}
Shu, F. H., 
Self-similar collapse of isothermal spheres and star formation. 
{\em Astrophys. J.} 
{\bf 214},
488--497 (1977).

\bibitem{WhSu1985}
Whitworth A.,  Summers, D.,  
Self-similar condensation of spherically symmetric self-
gravitating isothermal gas clouds. 
{\em Mon. Not. R. Astr. Soc.} 
{\bf 214}
1--25 (1985).

\end{thebibliography}
\end{document}